\documentclass[12pt,twoside]{article}
\usepackage{graphicx}
\usepackage{yhmath}
\usepackage{mathdots}
\usepackage[nottoc,numbib]{tocbibind}

\usepackage[all,cmtip]{xy}

 \usepackage{amsfonts}
\usepackage{amsthm}
\usepackage{amsmath,amscd}

\date{March 19, 2016}

\begin{document}

\centerline {\Large{\bf  Clustered  Hyperbolic Categories}}

\centerline{}

\centerline{}

\centerline{\bf {Ibrahim Saleh}}

\centerline{Email: ibrahim.saleh@uwc.edu}

 \newtheorem{thm}{Theorem}[section]
 \newtheorem{cor}[thm]{Corollary}
 \newtheorem{lem}[thm]{Lemma}
 \newtheorem{prop}[thm]{Proposition}
 \theoremstyle{definition}
 \newtheorem{defn}[thm]{Definition}
 \newtheorem{defns}[thm]{Definitions}
 \theoremstyle{remark}
 \newtheorem{rem}[thm]{Remark}
 \newtheorem{rems}[thm]{Remarks}
 \newtheorem{exam}[thm]{Example}
 \newtheorem{exams}[thm]{Examples}
 \newtheorem{conj}[thm]{Conjecture}
 \newtheorem{que}[thm]{Question}
 \newtheorem{rem and def}[thm]{Remark and Definition}
 \newtheorem{def and rem}[thm]{Definition and Remark}
 \newtheorem{corr}[thm]{Corollary of the Proof of Proposition 5.15}
 \numberwithin{equation}{section}
\newtheorem{IbrI}{Lemma}[section]
\newtheorem{chiral}[IbrI]{Definition}
\newtheorem{IbrII}[IbrI]{Lemma}
\newcommand{\field}[1]{\mathbb{#1}}

 \newcommand{\eps}{\varepsilon}
 \newcommand{\To}{\longrightarrow}
 \newcommand{\h}{\mathcal{H}}
 \newcommand{\s}{\mathcal{S}}
 \newcommand{\A}{\mathcal{A}}
 \newcommand{\J}{\mathcal{J}}
 \newcommand{\M}{\mathcal{M}}
 \newcommand{\W}{\mathcal{W}}
 \newcommand{\X}{\mathcal{X}}
 \newcommand{\BOP}{\mathbf{B}}
 \newcommand{\BH}{\mathbf{B}(\mathcal{H})}
 \newcommand{\KH}{\mathcal{K}(\mathcal{H})}
 \newcommand{\Real}{\mathbb{R}}
 \newcommand{\Complex}{\mathbb{C}}
 \newcommand{\Field}{\mathbb{F}}
 \newcommand{\RPlus}{\Real^{+}}
 \newcommand{\Polar}{\mathcal{P}_{\s}}
 \newcommand{\Poly}{\mathcal{P}(E)}
 \newcommand{\EssD}{\mathcal{D}}
 \newcommand{\Lom}{\mathcal{L}}
 \newcommand{\States}{\mathcal{T}}
 \newcommand{\abs}[1]{\left\vert#1\right\vert}
 \newcommand{\set}[1]{\left\{#1\right\}}
 \newcommand{\seq}[1]{\left<#1\right>}
 \newcommand{\norm}[1]{\left\Vert#1\right\Vert}
 \newcommand{\essnorm}[1]{\norm{#1}_{\ess}}
 \newcommand{\beq}{\begin{equation}}
\newcommand{\eeq}{\end{equation}}
\newcommand{\rarr}{\rightarrow}
\newcommand{\cA}{\mathcal{A}}
\newcommand{\cS}{\mathcal{S}}
\newcommand{\cC}{\mathcal{C}}
\newcommand{\cU}{\mathcal{U}}
\newcommand{\cR}{\mathcal{R}}
\newcommand{\RMod}{R\text{-Mod}}
\newcommand{\AMod}{A\text{-Mod}}
\newcommand{\Rep}{\text{Rep}}
\newcommand{\Aut}{\text{Aut}}
\newcommand{\XAut}{\xi\Aut}
\newcommand{\Rtzn}{R \{\theta, z, n \}}
\newcommand{\TxC}{(\Theta, \xi, \cC)}
\newcommand{\Chir}{\text{Chir}}

\tableofcontents

\begin{abstract} We  introduce  \emph{clustered hyperbolic categories}, which are constructed using a ``functorial" version of preseeds mutations called \emph{categorical  mutations}.  Every weyl preseed $p$ gives rise to a \emph{categorical preseed} $\mathcal{P}$ which generates a clustered hyperbolic category, that is generated by  copies of categories,  each one is equivalent  to the  category of representations of the Weyl cluster algebras $\mathcal{H}(p)$.  A  ``categorical realization" of Weyl cluster algebra is provided in the sense of   defining a  map $\mathbf{F}_{p}$ from any clustered hyperbolic category induced from  $p$ to the Weyl cluster algebra $\mathcal{H}(p)$ where image of $\mathbf{F}_{p}$ generates  $\mathcal{H}(p)$.

\end{abstract}



\section{Introduction}

Cluster algebras were introduced by S. Fomin and A. Zelevinsky in [10, 11, 12, 19, 2]. A cluster algebra is a commutative algebra with a distinguished set of generators called cluster variables and particular type of relations called mutations. A quantum  version was introduced in  [3] and [7, 8, 9].

Generalized Weyl algebras  were first introduced by  V. Bavula in [1] and separately  as  hyperbolic algebras by  A. Rosenberg in [17]. Their  motivation was to find a ring theoretical frame work to study the representations theory of some important ``small algebras"  such as the first Heisenberg algebra, Weyl algebras and the universal enveloping algebra of the Lie algebra $sl(2)$. In [17, 15], Rosenberg and Lunts introduced \emph{ hyperbolic categories} which  are basically generalizations of the categories of representations of generalized Weyl algebras.

In [18], we introduced Weyl cluster algebras which are  non-commutative algebras generated by cluster variables produced from cluster-like  structures which are formed, by mutations, from (possibly infinitely many) copies of generalized Weyl algebras.

Several attempts have been made to introduce  ``categorifications" for cluster algebras, taking into account the different ways of defining the notion of categorification. In [13, 14], cluster algebras of certain finite  types were realized as Grothendick rings of  categories of representations of some quantum affine algebras. Another type of categorification of cluster algebras  was introduced in [5], which is Caldero-Chapoton map. In [4], \emph{cluster category }$\mathcal{C}(Q)$ was introduced for any finite quiver $Q$ with neither loops nor two cycles. The Caldero-Chapoton map $X_{T}$ is a map from  $\mathcal{C}(Q)$ to the ring of Laurent polynomials over $\mathbb{Z}$ in the initial cluster variables associated to $Q$. It sends certain indecomposable objects in $\mathcal{C}(Q)$ to cluster variables such that its  image  generates the cluster algebra $A(Q)$.

 In this paper we provide a similar type of categorification for  Weyl cluster algebras. We introduce a categorical version of Weyl preseeds called \emph{categorical preseeds} and a ``functorial" version of preseeds mutations called \emph{categorical mutations}, see Definitions 3.5 and Definition 3.9 respectively. Every categorical preseed generates an ambient category, called \emph{mutation category}, which is generated by (possibly infinitely many) hyperbolic categories. \emph{Clustered hyperbolic categories}  are, by definition, the full subcategories of mutation categories  such that each object  appears in only one categorical preseed that is mutationally equivalent to the initial categorical preseed, Definition 3.14 (3). A technique of identifying clustered hyperbolic categories as subcategories of mutation categories, is provided through combinatorial tools introduced in this paper, called \emph{zigzag presentations}. Which is a presentation that encodes the relations between the expressions of the \emph{skew Laurent objects}, which are the categorical dual of the cluster variables. Clustered hyperbolic categories are introduced in this paper as   ``categorifications" of Weyl cluster algebras.  That is, we define a  map from each clustered hyperbolic category to its associated Weyl cluster algebra such that the image  of the map generates the Weyl cluster algebra. In the following we summarize the main  statements of this article.

\begin{thm}  Every mutation category contains a clustered hyperbolic category as a full subcategory.
\end{thm}

Every Weyl preseed $p$ gives rise to a \emph{categorical preseed} $\mathcal{P}$ which is used to generate a mutation category  $\mathcal{H(P)}$. A specific clustered hyperbolic category $\mathfrak{C}(\mathcal{W}_{0})$, as a subcategory of $\mathcal{H(P)}$, is introduced in Theorem 3.28.

\begin{thm} The clustered hyperbolic category  $\mathfrak{C}(\mathcal{W}_{0})$ is generated by  equivalent hyperbolic categories; each one of them is equivalent to the category  $\mathcal{H}(p)$-mod, where $\mathcal{H}(p)$ is the Weyl cluster algebra generated from $p$.
\end{thm}

\begin{thm}
Let  $\mathfrak{C}$ be a  clustered hyperbolic subcategory of $\mathcal{H(P)}$. Then there is a map $\mathbf{F}_{p}: \text{Obj.}\mathfrak{C}\longrightarrow \mathcal{H}(p) $  such that  image of $\mathbf{F}_{p}$ generates $\mathcal{H}(p)$.
\end{thm}

Full versions of Theorems 1.1-1.3 are available in Lemma 3.24, Theorem 3.28 and Theorem 3.29 respectively.

The paper is organized as follows. Section $2$ is devoted for basic definitions of Weyl cluster algebras. In Section 3, we introduce the notion of categorical preseeds, categorical mutations and clustered hyperbolic categories. In the same section we  introduce, \emph{hyperbolic objects} and the zigzag presentations and some of their properties are given in Proposition 3.20 and Lemma 3.22. In Theorem 3.28 we provide a relation between the clustered hyperbolic category $\mathfrak{C}(\mathcal{W}_{0})$ and the category of representations of its associated Weyl cluster algebra. In Theorem 3.29, we introduce a map from any clustered hyperbolic subcategory of the mutation category $\mathcal{H(P)}$ to the associated Weyl cluster algebra $\mathcal{H}(p)$.

 Throughout the paper, $K$ is a field of zero characteristic and the notation $[1, k]$ stands for the set $\{1,\ldots, k\}$. All our categories are small with non-empty sets of objects, $Obj.\mathcal{A}$ stands for the set of all objects of the category $\mathcal{A}$ and  $Mor._{\mathcal{A}}(M, M')$ denotes all morphisms in the category $\mathcal{A}$ from the object $M$ to the object $M'$. Let $D$ be an associative $K$-algebra with a non-trivial center $Z(D)$. Then, the group of all automorphisms of $D$ over the filed of zero characteristic  $K$ will be denoted by $Aut._{K} (D)$. The functor $Id_{\mathcal{A}}$ is the identical functor of the category $\mathcal{A}$.


\section{Weyl cluster algebras}

\subsection{Generalized Weyl algebras}

\begin{defn}[Generalized Weyl algebra (1, 17)] Let $D$ be an associative $K$-algebra with $\varepsilon =\{\varepsilon _{1},\ldots, \varepsilon _{n}\}$ be a fixed set of elements of the center of $D$ and $\theta =\{\theta _{1},\ldots, \theta _{n}\} $ be
a set of ring  automorphisms  such that $\theta_{i} (\varepsilon _{j})=\varepsilon _{j}$ for all $i\neq j$. \emph{The generalized Weyl algebra} of degree $n$, denoted by $D_{n}\{\theta, \varepsilon \}$, is defined to be the  ring extension of  $D$ generated by the  $2n$ indeterminates  $x_{1}, \ldots, x_{n}; y_{1}, \ldots,y _{n}$ modulo the commutation relations
\begin{equation}\label{}
    x_{i}r=\theta_{i}(r)x_{i} \ \ \text {and} \ \ ry_{i}= y_{i}\theta_{i}(r), \ \ \forall i\in [1,n], \ \forall r\in R,
   \end{equation}
\begin{equation}\label{}
  x_{i}y_{i}=\theta_{i}(\varepsilon _{i}), \ y_{i}x_{i}=\varepsilon _{i},  \ \forall i\in [1,n],
     \end{equation}
     and
\begin{equation}\label{}
   \  x_{i}y_{j}=y_{j}x_{i}, \ \ x_{i}x_{j}=x_{j}x_{i} \ \text{and} \ y_{i}y_{j}=y_{j}y_{i},\ \ \forall i\neq j \in [1,n].
\end{equation}
We warn the reader that $x_{i}y_{i}\neq y_{i}x_{i}$ in general.
\end{defn}

\begin{exam} [6, 1,  17] Let $A_{n}$ be the $n^{th}$ Weyl algebra  generated by  $x_{1}, \ldots, x_{n},\\y_{1}\ldots, y_{n}$ over  $K$ subject to the  relations
\begin{equation}\label{}
   x_{i}y_{i}-y_{i}x_{i}=1, \  \text{and} \ \ x_{i}x_{j}=x_{j}x_{i}, \ \ y_{i}y_{j}=y_{j}y_{i} \ \ \text{for} \ \ i \neq j, \ \forall i, j \in [1, n].
\end{equation}
 Let $\varepsilon _{i}=y _{i}x _{i}$ and $D$ be the ring of polynomials $K[\varepsilon _{1},\ldots, \varepsilon _{n}]$ and $\theta_{i}: R\rightarrow R$, induced by $\varepsilon _{i}\mapsto \varepsilon _{i}+1, \varepsilon _{j}\mapsto \varepsilon _{j}, j\neq i,\ \text{for all} \ i, j\in [1, n]$. It is known that $A_{n}$ is isomorphic to the generalized Weyl algebra $D_{n}\{\theta, \varepsilon \}$.
\end{exam}

\begin{exam}[17] The coordinate algebra $A(SL_{q}(2,k))$ of algebraic quantum group $SL_{q}(2, k)$ is the $K$-algebra generated by  $x, y, u$, and $v$ subject to the following relations
\begin{equation}\label{}
   qux=xu, \ \ qvx=xv, \ \ qyu=uy, \ \ qyv=vy, \ \ uv=vu, \ \ q\in K \backslash \{0\}
\end{equation}
\begin{equation}\label{}
   xy=quv+1, \ \  \text{and} \ \ yx=q^{-1}uv+1.
\end{equation}
$A(SL_{q}(2,k))$ is isomorphic to the generalized Weyl algebra $D_{1} \{\theta, \varepsilon \}$, where  $D$ is the algebra of polynomials $K[u,v]$; $\varepsilon =1+q^{-1}uv$ and $\theta$ is an automorphism of $D$, defined by $\theta (f(u,v))=f(qu,qv)$ for any polynomial $f(u,v)$.

\end{exam}

\subsection{Weyl cluster algebras}
This subsection provides a brief introduction to Weyl cluster algebras, introduced in [18]. We start with a simpler version of the definition of \emph{preseeds} [18, Definition 3.2]  which serves the purpose of this article.

\begin{defn} [Preseeds]

\begin{enumerate}
  \item

Let  $\mathbb{P}$ be a finitely  generated (free) abelian group, written multiplicatively,  with set of  generators
 \begin{equation}\label{}
   F=\bigcup^{n}_{i=1}F_{i} \ \ \text{where} \ \   F_{i}=\{f_{i1},\ldots,f_{im_{i}}\} \ \text{for some natural numbers} \ m_{i}, i \in [1, n].
 \end{equation}
 Let $\textit{R}=K[\mathbb{P}]$ be the  group ring of $\mathbb{P}$ over $K$.  Let $\mathcal{D}_{n}=R(t_{1}, \cdots, t_{n})$ be the skew-filed of rational functions in $n$ (commutative) variables  over $R$, where $t_{1}, \ldots, t_{n}$ do not necessarily commute with the elements of the coefficients ring $R$. However, we assume that $\mathcal{D}_{n}$ is an Ore domain.

 The set of algebraically independent rational functions  $X=\{x_{1}, \cdots, x_{n}\}$ which generate $ \mathcal{D}_{n}$ is called a \emph{cluster}  if the following condition is satisfied

   \begin{equation}\label{}
   x_{i}x_{j}=x_{j}x_{i} \ \ \text{and} \ \ x_{i}f_{jr}=f_{jr}x_{i},\ \ \text{for every} \  i \neq j\in [1, n],  \  \text{for all} \ \ r\in[1, m_{j}];
 \end{equation}

 Note that: For every $i\in [1, n]$ the variable $x_{i}$ does not necessarily commute with  elements from the set $F_{i}$.

  \item The triple $p=(X, \theta, \xi)$ is called a \emph{Weyl preseed} of rank $n$ in $\mathcal{D}_{n}$ if we have the following

   \begin{enumerate}
   \item $X=\{x_{1}, \ldots, x_{n}\}$ is a cluster in  $\mathcal{D}_{n}$;

      \item $\theta=\{\theta_{1}, \ldots, \theta_{n}\}$ be a set of $n$ automorphisms of $\textit{R}$ such that

       \begin{equation}\label{}
    x^{\pm 1}_{k}f=\theta_{i}^{\pm 1} (f)x^{\pm 1}_{i}, \ \ \forall f \in F_{i},  \forall i \in[1, n];
\end{equation}

\item $\xi=\{\xi_{1}, \ldots, \xi_{n}\}$ is a subset of $\textit{R}$ such that for every $i\in[1,n], \xi_{i}$ is a binomial (sum of two monomials) in the  elements of $F_{i}$. The set $\xi$ will be called the set of  \emph{exchange binomials} of $p$.

   \end{enumerate}

\end{enumerate}

\end{defn}

For information about Ore domains we refer to [16, 3]. In the following, we will omit the word Weyl from the expression Weyl preseeds and  all  preseeds are of rank $n$ unless stated otherwise. Also for simplicity we will use $\mathcal{D}$ for the skew-field of fractions instead of $\mathcal{D}_{n}$.

\begin{defn}[Preseeds   mutations] Let $p=(X,\theta , \xi)$  be a preseed  in  $\mathcal{D}$. For each  $k\in [1,n]$, two new triples  $\mu^{R}_{k}(p)=(\mu_{k}^{R}(X), \hat{\theta}_{k}, \xi)$ and  $\mu^{L}_{k}(p)=(\mu_{k}^{L}(X), \hat{\theta}_{k}, \xi)$ can be obtained from $p$ as follows

\begin{itemize}
  \item (Right mutation)
   \begin{equation}\label{}
  \mu^{R}_{k}(x_{i})=\left\{
    \begin{array}{ll}
      \xi_{i} x^{-1}_{i}, & i=k; \\
      x_{i}, & i\neq k.
    \end{array}
  \right.
 \end{equation}

  \item (Left mutation)

   \begin{equation}\label{}
  \mu^{L}_{k}(x_{i})=\left\{
    \begin{array}{ll}
      x^{-1}_{i}  \xi_{i}, & i=k; \\
      x_{i}, & i\neq k.
    \end{array}
  \right.
 \end{equation}

\item $\hat{\theta}_{k}=(\theta_{1},\ldots, \theta^{-1}_{k},\ldots, \theta_{n})$.

\end{itemize}

   \end{defn}

\begin{prop} [18] Let $p=(X,\theta , \xi)$  be a  preseed    in  $\mathcal{D}$. Then the following are true
\begin{enumerate}
\item
  For any sequence of right mutations (respectively left)  $\mu^{R}_{i_{1}}\mu^{R}_{i_{2}}\ldots \mu^{R}_{i_{q}}$, we have
     $\mu^{R}_{i_{1}}\mu^{R}_{i_{2}}\ldots \mu^{R}_{i_{q}}(p)$  (respectively $\mu^{L}_{i_{1}}\mu^{L}_{i_{2}}\ldots \mu^{L}_{i_{q}}(p)$) is again a  preseed;
 \item For every $k \in [1, n]$,
 \begin{equation}\label{}
    \mu^{R}_{k} \mu^{L}_{k}(p)=\mu^{L}_{k}\mu^{R}_{k}(p)=p.
 \end{equation}
\end{enumerate}
\end{prop}

\begin{defn}[Cluster sets and exchange graphs]
\begin{enumerate}

\item
Let $p$ be a preseed  in $\mathcal{D}$. An element $y \in \mathcal{D}$ is said to be a \emph{cluster variable} of $p$  if $y$  is  an element in some  cluster $Y$ of some preseed $s=(Y, \theta, \xi)$ which is obtained from  $p$  by applying some sequence of (right or left) mutations. The set of all  cluster variables of $p$  is called   \textit{the  cluster set} of $p$  and is denoted by  $\mathcal{X }(p)$. The elements of the cluster of $p$ are called \emph{initial cluster variables}.

\item The \emph{exchange graph} of a preseed $p$, denoted by $\mathbb{G}(p)$, is the $n$-regular graph whose vertices are labeled by the preseeds that can be obtained from $p$ by applying some sequence of right or left mutations and whose edges correspond to mutations. Two adjacent preseeds in $\mathbb{G}$ can be obtained from each other by applying right mutation $\mu^{R}_{k}$ or left mutation $\mu^{L}_{k}$   for some $k\in [1, n]$.
\end{enumerate}

\end{defn}

\begin{exam} Let $p$ be the  rank 1 preseed $(\{x_{1}\},\theta_{1}, \xi_{1})$ with $F_{1}=\{f_{1}\}$, $\xi_{1}=1+m(f_{1})$ and $m(f_{1})$ is a monomial in $f_{1}$,  $\theta_{1}$ is an $R$-automorphism, where $\textit{R}=K[ f^{n}_{1}; n \in \mathbb{Z}]$ and $\mathcal{D}=\textit{R}(x_{1})$. Applying mutation at   $x_{1}$ produces the following cluster variables

\begin{eqnarray*}
   x_{1}&\stackrel{\mu^{R} _{k}}{\Rightarrow}& \xi_{1}x_{1}^{-1} \\
   &\stackrel{\mu^{R} _{k}}{\Rightarrow}& \xi_{1}x_{1}\xi_{1}^{-1} \\
      &\stackrel{\mu^{R} _{k}}{\Rightarrow}& \xi_{1}^{2}x^{-1}_{1}\xi_{1}^{-1} \\
      &\stackrel{\mu^{R} _{k}}{\Rightarrow}& \xi_{1}^{2}x_{1}\xi_{1}^{-2} \\   &\ldots &\\
  &\stackrel{\mu^{R} _{k}}{\Rightarrow}&\xi_{1}^{k+1}x^{-1}_{1}\xi_{1}^{-k}\\
  &\stackrel{\mu^{R} _{k}}{\Rightarrow}&\xi_{1}^{k+1}x_{1}\xi_{1}^{-(k+1)}\\
   &\ldots &,\\
\end{eqnarray*}
and
\begin{eqnarray*}
 x_{1}&\stackrel{\mu^{L} _{k}}{\Rightarrow}& x_{1}^{-1}\xi_{1}\\
   &\stackrel{\mu^{L} _{k}}{\Rightarrow}& \xi_{1}^{-1}x_{1}\xi_{1} \\
      &\stackrel{\mu^{L} _{k}}{\Rightarrow}& \xi_{1}^{-1}x^{-1}_{1}\xi_{1}^{2} \\
      &\stackrel{\mu^{L} _{k}}{\Rightarrow}& \xi_{1}^{-2}x_{1}\xi_{1}^{2} \\
   &\ldots &\\
  &\stackrel{\mu^{L} _{k}}{\Rightarrow}&\xi_{1}^{-k}x^{-1}_{1}\xi_{1}^{k+1}\\
  &\stackrel{\mu^{L} _{k}}{\Rightarrow}&\xi_{1}^{-(k+1)}x_{1}\xi_{1}^{k+1}\\
   &\ldots &.\\
\end{eqnarray*}

So we have the  infinite  cluster  set
\begin{equation}\label{}
 \nonumber \mathcal{X } (p)=\{x_{1}, \xi_{1}^{k+1}x_{1}^{-1}\xi_{1}^{-k}, \xi_{1}^{k}x_{1}\xi_{1}^{-k},\xi_{1}^{-k}x_{1}^{-1}\xi_{1}^{k+1},\xi_{1}^{-k}x_{1}^{-1}\xi_{1}^{k}, k \in \mathbb{Z}\}.
\end{equation}

\end{exam}

In the following example we will see that every generalized Weyl algebra gives rise to a preseed.

\begin{exam} Let $D_{n} \{\theta, \xi\}$ be a generalized  Weyl algebra.  Consider the triple   $p=(Y, \xi, \theta )$, where the set of exchange binomials  $\xi=\{\xi_{i};  i\in [1, n]\}$ and $Y=\{y_{1}, \ldots, y_{n}\}$. From the properties of the $R$-automorphisms $\theta=(\theta_{1}, \ldots, \theta_{n})$  given in Equations (2.1) and (2.2)  one can see that $\theta_{i}$ satisfies Equation (2.9)   for each $i\in [1, n]$ which makes $p$  a preseed in   $\mathcal{D}$, where  $\mathcal{D}$ is the division  ring of rational functions in $y_{1}, \ldots, y_{n}$ over the ring $\textit{R}=K[\mathbb{P}]$. In particular, in the case of  the $n^{th}$ Weyl algebra  $A_{n}$, one can see that the  skew-field of rational functions $\mathcal{D}=R(y_{1}, \ldots, y_{n})$  is an Ore domain.
\end{exam}

\begin{exam} Recall  the coordinate algebra $A(SL_{q}(2,k))$ of the algebraic quantum group $SL_{q}(2, k)$, Example 2.3. Let $F_{1}=\{qu, v\}$.  Consider the rank 1 preseed  $p=(\{x\}, \{\theta\}, \{\zeta\})$, where  $\theta: R\rightarrow R$ given by  $\theta (f(u,v))=f(qu,qv)$ and $\zeta=quv+1$. One can see that $p$ is a  preseed in the division ring $\mathcal{D}=K[\mathbb{P}](x)$, where $\mathbb{P}$ is the free abelian group generated by $F_{1}$.   The cluster set of $p$ is given by

\begin{equation}\label{}
  \mathcal{X }(p)=\{x,\zeta^{j}x\zeta^{-j}, \zeta ^{j+1}x^{-1}\zeta ^{-j-1}, j\in
\mathbb{N} \}\cup\{y,\zeta^{j}y\zeta^{-j}, \zeta ^{j+1}y^{-1}\zeta ^{-j-1},
j\in \mathbb{N}\}
\end{equation}

\end{exam}

\begin{defn}[Weyl cluster algebras]  Let  $p=(X, \xi, \theta)$  be a preseed in $\mathcal{D}$. The Weyl cluster algebra $\mathcal{H}(p)$  is defined to be the  $\textit{R}$-subalgebra of $\mathcal{D}$  generated by the cluster set  $\mathcal{X} (p)$.
\end{defn}

The following remark and theorem shed some light on the structure of the Weyl cluster algebra $\mathcal{H}(p)$. Remark 2.12 and first part of  Theorem 2.13 can be phrased as following: The Weyl cluster algebra $\mathcal{H}(p)$ is generated by $R$ and many (could be infinitely many) isomorphic copies of generalized Weyl algebras, each vertex in the exchange graph of $p$ gives rise to two copies of them.

\begin{rem and def} Let $p=(X, \theta, \xi)$ be a Weyl preseed and $D=K[\xi_{1}, \ldots, \xi_{n}]$ be the ring of polynomials in $\xi_{1}, \ldots, \xi_{n}$ where $\xi_{i}, i=1,\ldots, n$ are as defined in Example 2.9. Then $p$ gives rise to two copies of generalized Weyl  algebras of rank $n$, as follows
\begin{enumerate}
  \item [(a)] $H^{R}(p)$ is the ring extension of $R$ generated by $\mu^{R}_{1}(x_{1}), \ldots, \mu^{R}_{n}(x_{n}), x_{1}, \ldots, x_{n}$.
  \item [(b)] $H^{L}(p)$ is the ring extension of $R$ generated by $x_{1}, \ldots, x_{n}, \mu^{L}_{1}(x_{1}), \ldots, \mu^{L}_{n}(x_{n})$.
  \item [(c)] In particular, if $p=(X, \xi, \theta)$ is the preseed given in Example 2.9, then each of $H^{R}(p)$ and $H^{L}(p)$ is isomorphic to $D_{n}\{\theta, \xi\}$ as generalized Weyl algebras.
\end{enumerate}

\end{rem and def}

\begin{thm} Let $p=(X, \xi, \theta)$ be a Weyl preseed in $\mathcal{D}$. Then the following are true
\begin{enumerate}

\item Right and left mutations on $p$ induce isomorphisms between the generalized Weyl algebras $H^{R}(p)$ and $H^{R}(\mu^{R}_{k}(p))$ (respectively $H^{L}(p)$ and $H^{L}(\mu^{L}_{k}(p))$).
\item The  Weyl cluster algebra $\mathcal{H}(p)$ is a subring of the (non-commutative)  ring of Laurent polynomials in the initial exchange cluster variables with coefficients from ring of polynomials $R[\theta^{\pm 1}_{1}(\xi^{-1}_{1}), \ldots,  \theta^{\pm 1}_{n}(\xi^{-1}_{n})]$.
\item The Weyl cluster algebra $\mathcal{H}(p)$ is finitely generated and is isomorphic to each of $H^{R}(s)$ and $H^{L}(s)$  for every preseed $s$ mutationally equivalent to $p$.
\end{enumerate}
\end{thm}

\begin{proof} Proofs of Parts (1) and (2) are provided in [18, Theorem 4.12].\\
To prove Part (3), we only need to prove that  the generators of the algebra $\mathcal{H}(p)$ are also elements in the algebras $H^{R}(s)$ and $H^{L}(s)$. The algebra $\mathcal{H}(p)$ is generated by the set of all cluster variables that are obtained from the initial cluster variables $x_{1}, \ldots, x_{n}$. We will show that every cluster variable generated from $x_{k}, k=1, \ldots, n$ by applying some sequence of mutations is already an element of $H^{R}(s)$ and $H^{L}(s)$. From the proof of Part (3) in Theorem 4.12 in [18], a cluster variable $y$ that obtained from  $x_{k}$ by applying a sequence of mutations (right or left) of length $l$, can be written in the form

\begin{equation}\label{}
   y=\left\{
                  \begin{array}{ll}
                    \xi_{k}^{\frac{l+1}{2}}x_{k}^{-1}\xi_{k}^{-(\frac{l+1}{2}-1)} \ \text{or} \ \ \xi_{k}^{-(\frac{l+1}{2}-1)}x_{k}^{-1}\xi_{k}^{\frac{l+1}{2}}, \  \text{if} \ l  \ \text{is an odd number;}& \\
                     \xi_{k}^{\frac{l}{2}}x_{k}\xi_{k}^{-\frac{l}{2}} \ \text{or} \  \  \xi_{k}^{-\frac{l}{2}}x_{k}\xi_{k}^{\frac{l}{2}}, \ \   \text{if} \ l \ \text{is an even number.}&
                  \end{array}
                \right.
\end{equation}

Assume that the sequence of mutations that creates $s$ from $p$ contains $m$ copies of $\mu^{R}_{k}$ (respect to $\mu^{L}_{k}$), then the set of generators of the hyperbolic algebra $H^{R}(s)$ (respect to $H^{L}(s)$) contains elements of the form  $\xi_{k}^{\frac{m+1}{2}}x_{k}^{-1}\xi_{k}^{-(\frac{m+1}{2}-1)}$ and $\xi_{k}^{\frac{m}{2}}x_{k}\xi_{k}^{-\frac{m}{2}}$ (respect to $\xi_{k}^{-(\frac{m+1}{2}-1)}x_{k}^{-1}\xi_{k}^{\frac{m+1}{2}}$ and $ \xi_{k}^{-\frac{m}{2}}x_{k}\xi_{k}^{\frac{m}{2}})$. One can see that whether $m \geq l$ or $m< l$, $y$ can be obtained from the generators of $H^{R}(s)$  (respect to $H^{L}(s)$) by multiplying them from left and right by $\xi^{\pm q}_{k}$ for some natural number $q$.

\end{proof}

\begin{exam}[Weyl cluster algebra associated to first Weyl algebra] Recall the $n^{th}$ Weyl algebra  and the associated preseed  given in Example 2.9. Let $A_{1}$ be the first Weyl algebra and consider the preseed
$p_{1}=(\{y\}, \{\varepsilon\}, \theta \})$. Here $R=K [\mathbb{P}]$, where $\mathbb {P}$ is the cyclic group generated by $\varepsilon=yx$. Then We have the following exchange graph
\begin {itemize}
\item $\mathbb{G}(p_{1})$
\begin{equation}\label{}
    \xymatrix{ \ldots  \ar[r]_{L}&\ar[r]^{R} \stackrel{y_{-3}}{\cdot} \ar[l]_{R}&\ar[r]^{R} \stackrel{y_{-2}}{\cdot} \ar[l]^L &\ar[r]^{R} \stackrel{y_{-1}}{\cdot}\ar[l]^{L} &\ar[r]_{L} \stackrel{y_{0}=y}{\cdot}\ar[l]^{L} &\ar[l]_{R}\stackrel{y_{1}}{\cdot}\ar[r]^{R} &\ar[l]^L\stackrel{y_{2}}{\cdot}\ar[r]^R&\ar[l]^L\stackrel{y_{3}}{\cdot}\ar[r]^R & \ar[l]^{L} \ldots
    },
\end{equation}
\end {itemize}
(here $ \xymatrix{ & {\cdot}\ar[l]^{L}}$ is left mutation and $
\xymatrix{{\cdot}\ar[r]^R &}$ is right mutation). Which can be encoded by
the following equations
\begin{equation}\label{}
 y_{2k}y_{2k\pm 1}=y_{2k\pm 1}y_{2k}+1, \ \ \ \text{for} \ \  k\in \mathbb{Z}.
\end{equation}
The  Weyl cluster algebra  $\mathcal{H}(p_{1}(y))$ is the $R$-subalgebra of $\mathcal{D}$ generated by the set of cluster variables $\{y_{k}, k\in \mathbb{Z}\}$. Relations (2.16) can be interpreted as follows,  each double heads arrow in $\mathbb{G}(p_{1})$ corresponds to a copy of  first Weyl algebra, denoted by $A^{k}_{1}=K\langle y_{k}, y_{k+1}\rangle, k\in \mathbb{Z}$ and right (respectively left) mutations define isomorphisms between the adjacent copies, given by $T_{k}:A^{k}_{1}\rightarrow A^{k+1}_{1},  y_{k}\mapsto y_{k+1} $ for $k\in \mathbb{Z}$ (respectively to the inverses of $T_{k}, k\in \mathbb{Z}$).
\end{exam}

\section{Clustered hyperbolic categories}

\subsection{Hyperbolic category}
For a category $\mathcal{A}$, an auto-equivalence $\Phi$ is an invertible functor on  $\mathcal{A}$ making it equivalent to itself.  A natural transformation $\tau: Id_{\mathcal{A}}\rightarrow Id_{\mathcal{A}}$ is said to be an endomorphism of the identical functor of $\mathcal{A}$ if for every object $W$ there is a morphism $\tau_{W}\in Hom_{\mathcal{A}}(W, W)$ such that for any $W'$ in $Obj.\mathcal{A}$ and every $f \in Hom_{\mathcal{A}}(W, W')$ we have $\tau_{W'} \circ f=f \circ \tau_{W}$. Consider the following examples
\begin{enumerate}
  \item [(a)] Let $D$ be an associative $K$-algebra with non-trivial center. Then every automorphism $\phi$ of $D$ gives rise to an auto-equivalence $\Phi$ of the category $D$-mod of all modules over $D$,  where the module structure on  $\Phi(W)$ is given by
\begin{equation}\label{}
 \nonumber  \ \ d \cdot_{\Phi(W)}w=\phi(d)\cdot_{W}w, \ \ d \in D,  w \in W.
\end{equation}
  \item [(b)] Every element $\varsigma \in Z(D)$ defines an endomorphism $\bar{\varsigma}$ of the identical functor of the category $D$-mod,  as follows

\begin{equation}\label{}
 \nonumber \bar{\varsigma}_{W}:W \rightarrow W, \ \text{given by} \ \ \bar{\varsigma}(w)=\varsigma\cdot w.
\end{equation}
\end{enumerate}

\begin{defn}[Hyperbolic category (15, 17)] Let $\mathcal{A}$ be an additive category with a set of $n$-auto-equivalences
$\Theta =\{\Theta _{1},\ldots, \Theta _{n}\}$ and another set of $n$-endomorphisms $\xi=\{\xi_{1},\ldots, \xi_{n}\}$ of the identical functor of $\mathcal{A}$. Consider the endomorphism  $\varepsilon_{i}$ of the identical functor of $\mathcal{A}$ given by
\begin{equation}\label{}
\varepsilon_{i}:=\Theta^{-1}_{i} (\xi_{i}), \ \  i\in [1, n].
\end{equation}\textit{The hyperbolic category} of rank $n$ on  $\mathcal{A}$, denoted by  $\mathcal{A}_{n}\{\Theta, \xi\}$,  is defined as follows:
 Objects  are triples $(\gamma, M, \eta)$  where $M$ is an object in $\mathcal{A}$ and $\gamma$ and $\eta$ are two sets of   $\mathcal{A}$-morphisms   such that  $\gamma=\{\gamma _{1},\ldots, \gamma _{n}\}$,
  $\eta =\{\eta _{1},\ldots, \eta_{n}\}$ where
\begin{equation}\label{}
   \nonumber \gamma _{i}: M\longrightarrow \Theta_{i} (M) \ \ \ \text{and} \ \ \ \eta _{i}:\Theta_{i} (M)\longrightarrow M, \ \  i \in [1,n]
\end{equation}
given by
\begin{equation}\label{}
    \eta _{i} \circ \gamma _{i}=\xi _{i, M} \ \ \ \text{and} \ \  \ \gamma _{i}\circ \eta _{i}=\varepsilon_{i,\Theta_{i} (M)}, \ \ \ \forall i \in [1, n].
\end{equation}

A  morphism from the object  $(\gamma, M, \eta)$ to the object $(\gamma', M',
\eta')$ is  the set  $f=\{f_{1}, \ldots, f_{n}\}$ where for  $i\in [1,n]$,  $f_{i}$ is in $Mor._{\mathcal{A}}(M, M')$, such that  the following diagram is commutative
\begin{equation}\label{}
   \xymatrix{
 M \ar[d]^{f_{i}} \ar[r]^{\gamma_{i}} & \ar[d]^{\Theta_{i}(f_{i})}\Theta_{i} (M)\ar[r]^{\eta _{i}}& \ar[d]^{f_{i}} M\\
  M'  \ar[r]^{\gamma'_{i}}  &  \Theta_{i} (M')\ar[r]^{\eta'_{i}} & M'.}
\end{equation}

\end{defn}

\begin{rem}  Let $M$ be an object in $\mathcal{A}$ and $id_{M}$ be its identity morphism. Then the $n$-tuple $(id_{M}, \cdots, id_{M})$ is  a morphism in the hyperbolic category $\mathcal{A}_{n}\{\Theta, \xi\}$ if and only if  the object $(\gamma, M, \eta) \in Obj.\mathcal{A}_{n}\{\Theta, \xi\}$  has unique sets of  $\mathcal{A}$-morphisms $\gamma$ and $\eta$ as defined  in (3.2) above. In particular, the object $M$ in $\mathcal{A}$ will appear exactly once as the object $(-, M, -)$  in   $\mathcal{A}_{n}\{\Theta, \xi\}$.
\end{rem}

\begin{exam} Let $\mathrm{D}_{1}\{\theta, \xi\}$ be a generalized Weyl algebra  with indeterminants $x$ and $y$. Let $\mathcal{A}$ be the full subcategory of the category $\mathrm{D}-mod$ with objects are $\mathrm{D}_{1}\{\theta, \xi\}$-modules, forgetting about the  actions of $x$ and $y$.  The category  $\mathrm{D}_{1}\{\theta, \xi\}$-modules is equivalent to a  hyperbolic category  $\mathcal{A}_{1}\{\Theta, \xi\}$ where $\Theta: \mathcal{A}\longrightarrow \mathcal{A}$ is induced by the  $\mathrm{R}$-automorphism $\theta$ and for $M \in$ objets of $\mathcal{A}$,  $\xi_{M}:M\longrightarrow M$ is given by $\xi_{M}(m):=\xi m$. Objects of  $\mathcal{A}_{1}\{\Theta, \xi\}$ are the triples $(\overline{x}, M, \overline{y})$ where $M$ is an object in $\mathcal{A}$,
 $\overline{x}:M\longrightarrow \theta (M)$ given by $\overline{x}(m)=xm$
 and $\overline{y}:\theta (M)\longrightarrow M$ given by
 $\overline{y}(m)=ym$.
\end{exam}

\begin{exam}[Trivial hyperbolic category] Let $D_{n}\{\theta, \xi\}$ be a generalized Weyl algebra. Consider the  additive category $\mathcal{B}$ with only one object which is the ring $D$ and  $Hom_{\mathcal{B}}(D, D)$  is $D$. One can form a hyperbolic category $\mathcal{\dot{B}}(\Theta, \xi)$ with one object which is $(\overline{x}, D, \overline{y})$, where $\overline{x}=\{x_{1},\cdots, x_{n}\}$ and $\overline{y}=\{y_{1},\cdots, y_{n}\}$ and $\Theta$ is induced from $\theta$.
\end{exam}

\subsection{Categorical preseeds}
Let $\mathcal{A}$ be an additive category.
\begin{defn} [Categorical preseeds]  A \emph{categorical preseed} of rank $n$ in $\mathcal{A}$ is the  data $\mathcal{P}=(\Theta, \xi, \mathcal{A}\langle\Theta\rangle)$ where
\begin{enumerate}
\item $\Theta=\{\Theta _{1}, \ldots, \Theta _{n}\}$ is a set of  $n$  auto-equivalences in $\mathcal{A};$
\item $\xi=\{\xi_{1},\ldots, \xi_{n}\}$ is set of  $n$  endomorphisms of the identical functor of $\mathcal{A}$;
\item $\mathcal{A}\langle\Theta\rangle$ is the following category:

 Objects are  pairs $(M, f)$, where $M$ is an object in $\mathcal{A}$ and $f=\{f _{1},\ldots,f _{n}\}$ is
    a set of  $n$ invertible elements of $Mor._{\mathcal{A}}(\Theta (M), M)$.

Morphisms from $(M, f)$ to $(M', f ')$ in $\mathcal{A}\langle\Theta\rangle$  are the $n$-tuples $h=\{h_{1},\ldots , h_{n}\}$ where $h_{i}\in  Mor._{\mathcal{A}}(M ,M'), i\in [1, n]$ such that
the following diagram is commutative
\begin{equation}\label{}
    \xymatrix{ \ar[d]_{\Theta_{i}(h_{i})}\Theta_{i} (M)\ar[r]^{f _{i}}& \ar[d]^{h_{i}} M\\
    \Theta_{i} (M')\ar[r]^{f'_{i}} & M'.}
\end{equation}

\end{enumerate}
\end{defn}

\begin{rem} \begin{enumerate} \item Let  $\mathcal{P}=(\Theta, \xi, \mathcal{A}\langle\Theta\rangle)$  be a categorical preseed. Since   $\xi$ is a set of endomorphisms of $Id_{\mathcal{A}}$, then for every object $(M, f)$ in $\mathcal{A}\langle\Theta\rangle$ we have

\begin{equation}\label{}
    f _{i} \circ \xi _{i, \Theta_{i} (\Theta _{j_{k}}\ldots \Theta _{j_{1}}(M))}=\xi_{i,\Theta _{j_{k}}\ldots \Theta _{j_{1}}(M)}\circ f _{i}; \ \ \ \forall i,   j_{k},\ldots, j_{1}\in [1,n].
\end{equation}

\item  The category $\mathcal{A}\langle\Theta\rangle $, defined in Definition 3.5, was  introduced in [17] under the name of \emph{skew Laurent category}. The objects of  $\mathcal{A}\langle\Theta\rangle $ will be called the \emph{skew Laurent objects} of $\mathcal{P}$.

\end{enumerate}

\end{rem}

\begin{exam} Let $D$ be an associative ring with an automorphism $\theta$. Consider the ring $D[x, x^{-1}]$ of skew Laurent polynomials in $x$, subject to the relation
\begin{equation}\label{}
 \nonumber  rx=x\theta(r) \ \ \text{for every} \   r \in D.
\end{equation}

 The category $D[x, x^{-1}]$-mod is equivalent to the skew Laurent  category $(D-mod)\langle\Theta\rangle$ where  $\Theta$ is the automorphism functor of $D$-mod induced by $\theta$. The objects of the category $(D-mod)\langle\Theta\rangle$ are the pairs $(M, \overline{x})$ where $M$ is an object in $D$-mod and $\overline{x}:\Theta (M)\rightarrow M$ given by $\overline{x}(m)=xm$. Moreover, we can define a  categorical preseed of rank $1$ in the category $D$-mod as follows, let $z$ be  a nonunit element in the center of $D$, so for every $M$, an object of $D[x, x^{-1}]$-mod, let $\xi_{M}:M\rightarrow M$ be the module homomorphism  given by $\xi(m)=zm$, for any $m$ in $M$. One can see that the triple  $\mathcal{P}=(\Theta, \xi, \mathcal{A}\langle\Theta\rangle)$ is a categorical pressed of rank 1.
\end{exam}

\begin{exam}[\emph{Weyl categorical preseed}] Let $D_{n}\{\theta, \xi\}$ be a generalized Weyl algebra  in  the indeterminates  $x_{1},\ldots,x_{n}; y_{1},\ldots,y_{n}$ with $\theta =\{\theta_{1},\ldots,\theta_{n}\}\subset Aut(D)$ and $\{z_{1},\ldots, z_{n}\}\subset Z(D)$, as in Definition 2.1. Let $A=S^{-1}\mathrm{D}[y_{1}^{\pm 1},\ldots,y_{n}^{\pm 1}]$, where $S=\{z_{i}, \theta_{i} (z_{i}); \  i\in [1,n]\}$. Let $\mathfrak{A}=D-mod$, and $\mathcal{A}=\varphi_{*}(\mathfrak{A})$, where $\varphi_{*}:\mathfrak{A}\longrightarrow \mathrm{R}-mod$, the functor that sends each object in $\mathfrak{A}$ to itself as a $D$-module forgetting the rest of the actions of $A$. Let $\mathcal{W}=(\Theta, \xi, \mathcal{A}\langle\Theta\rangle)$ such that $\Theta=\{\Theta _{1},\ldots, \Theta _{n}\}$ is a set of  $n$ $\mathcal{A}$-auto equivalences where $\Theta _{i}:\mathcal{A}\longrightarrow \mathcal{A}$ is induced by $\theta _{i}$; and $\xi=\{\xi_{1},\ldots,\xi_{n}\}$ is a set of  $n$ endomorphisms of the identity functor of  $\mathcal{A}$ given by $\xi _{i W}: W\longrightarrow W$ where $\xi_{i W}(w)=z_{i}\cdot w, i\in [1,n]$, for any object $W$ in $\mathcal{A}$.

 Objects of $\mathcal{A}\langle\Theta\rangle$ are  pairs $(M, \overline{y})$ such that $M$ is an object of  $\mathcal{A}$ and $\overline{y}=\{\overline{y}_{1},\ldots, \overline{y}_{n}\}$ where  $\overline{y}_{i}:\Theta_{i} (M)\longrightarrow M$, $\overline{y}_{i}(m)=y_{i}\cdot m$ for $i\in [1,n]$ and morphisms of $\mathcal{A}\langle\Theta\rangle$ are given by $Mor._{\mathcal{A}\langle\Theta\rangle}((M,\overline{y}),(M',\overline{y}))=Mor._{\mathfrak{A}}(M,M')$.  This specific hyperbolic categorical pressed $\mathcal{W}$ is called Weyl categorical preseed.

\end{exam}
\subsection{Categorical mutations}

The main purpose of this subsection is to introduce \emph{categorical mutations} which are  ``functorial" versions of the preseeds mutations introduced in [18]. Categorical mutations involve creating a new skew Laurent category from a previous one and a functor between the old and the new   categories.

For any morphism $h$  in $Mor._{\mathcal{A}}(W, W')$ we write $D(h)=W$ and $C(h)=W'$, where $D(h)$ stands for the domain of $h$ and $C(h)$ is the codomain of $h$.

 Let $\mathcal{P}=(\Theta, \xi, \mathcal{A}\langle\Theta\rangle)$ be a categorical preseed of rank $n$ in $\mathcal{A}$. In the following we  introduce length one and length two sequences of categorical mutations before introducing the general rules of categorical mutations.

\begin{defn}[\textbf{\emph{Categorical  mutations}}]
\begin{enumerate}

 \item \textbf{Categorical right mutations.}
\begin{enumerate}
 \item[(a)] \textbf{First generation  categorical right mutations.} For $k\in [1, n]$, the action of \emph{categorical right mutation} $\mu^{R}_{k}$ on $\mathcal{P}$ is given by $\mu^{R}_{k}(\mathcal{P})=(\widehat{\Theta}_{(k)} ^{-1},\xi ,\mathcal{A}_{k^{(1)}}\langle \Theta \rangle)$,  where

      \begin{itemize}
    \item $\widehat{\Theta}_{(k)} ^{-1}=(\Theta _{1}, \ldots, \Theta ^{-1}_{k},\ldots, \Theta _{n})$;

    \item Every object in $\mathcal{A}_{k^{(1)}}\langle\Theta\rangle$ is a mutated object from $\mathcal{A}\langle\Theta\rangle$ as follows

     \begin{equation}\label{}
  \nonumber (M, f) \in   \ Obj.\mathcal{A}\langle\Theta\rangle \ \text{if and only if} \ (\Theta_{k} (M), \mu^{R}_{k}(f_{k}))\in \ \ Obj.\mathcal{A}_{k^{(1)}}\langle\Theta\rangle,
     \end{equation}
     where
     \begin{equation}\label{}
       \mu^{R}_{k}(f)=(f_{1}, \ldots, f_{k-1}, \varepsilon_{k,D(f_{k})}\circ f_{k} ^{-1}, f_{k+1}, \ldots, f_{n})
     \end{equation}

  \item The morphisms $\mathrm{Mor.}_{\mathcal{A}_{k^{(1)}}\langle\Theta\rangle}((\Theta _{k}(M),\mu_{k}(f)),(\Theta _{k}(M'),\mu_{k}(f')))$ are given by $\Theta_{k}(\mathrm{Mor.}_{\mathcal{A}\langle\Theta\rangle}((M,f),(M',f ')))$ which is
\begin{equation}\label{}
\{\Theta_{k}(h); h\in \mathrm{Mor.}_{\mathcal{A}\langle\Theta\rangle}((M,f),(M',f ')) \}.
\end{equation}
     \end{itemize}

 \item[(b)] \textbf{Second generation categorical right mutations.} The second mutation $\mu^{R}_{k^{2}}(\mathcal{P})$   alters $\mathcal{A}_{k^{(1)}}\langle\Theta\rangle$ to the category $\mathcal{A}_{k^{(2)}}\langle\Theta\rangle$  with objects given by

\begin{equation}\label{}
  \textbf{(}M, (f_{1}, \ldots, f_{k-1},\xi_{k, C(f_{k})}\circ f_{k} \circ \varepsilon^{-1}_{k,D(f_{k})} , f_{k+1}, \ldots, f_{n})\textbf{)},
\end{equation}

\end{enumerate}
  \item \textbf{Categorical left  mutations.}
    \begin{enumerate}
    \item[(a)] \textbf{First generation categorical left mutation} $\mu^{L}_{k}(\mathcal{P})$ is given by $\mu^{L}_{k}(\mathcal{P})=(\widehat{\Theta}_{(k)}^{-1}, \xi, \mathcal{A}_{k^{(-1)}}\langle\Theta\rangle)$, where

       \begin{itemize}
    \item Every object in $\mathcal{A}_{k^{(-1)}}\langle\Theta\rangle$ is a mutated object from $\mathcal{A}\langle\Theta\rangle$ as follows

     \begin{equation}\label{}
   \nonumber (M, f) \in  \  Obj.\mathcal{A}\langle\Theta\rangle \ \text{if and only if} \ (\Theta_{k} (M), \mu^{L}_{k}(f_{k}))\in \  Obj.\mathcal{A}_{k^{(-1)}}\langle\Theta\rangle,
     \end{equation}
      where
     \begin{equation}\label{}
       \mu^{L}_{k}(f)=(f_{1}, \ldots, f_{k-1},f_{k}^{-1}\circ \varepsilon_{k, C(f_{k})}, f_{k+1}, \ldots, f_{n}).
     \end{equation}
\item
     The morphisms of $\mathcal{A}_{k^{(-1)}}\langle\Theta\rangle$ are as given  in (3.7).
     \end{itemize}

\item[(b)] \textbf{Second generation categorical left mutations.} The second mutation $\mu^{L}_{k^{2}}(\mathcal{P})$   alters $\mathcal{A}\langle\Theta\rangle$  to $\mathcal{A}_{k^{(-2)}}\langle\Theta\rangle$ with objects given by
\begin{equation}\label{}
  \textbf{(}M, (f_{1}, \ldots, f_{k-1},\varepsilon^{-1}_{k, C(f_{k})}\circ f_{k} \circ \xi_{k,D(f_{k})} , f_{k+1}, \ldots, f_{n})\textbf{)}.
\end{equation}
\end{enumerate}

\item   \textbf{General categorical  mutations rules.} Let $\mathcal{S}=(\Theta, \xi, \mathcal{A}_{k^{(t)}}\langle \Theta \rangle)$ be a triple obtained from the initial  categorical preseed $\mathcal{P}$ by applying some  sequence of categorical mutations $\mu$ which contains  exactly $t$ copies of $\mu^{R}_{k}$  with $t\geq 2$. Applying another $\mu^{R}_{k}$ on $\mathcal{S}$ is governed by the following rules
\begin{enumerate}
\item $\xi$ is  frozen.
\item $\Theta $ is altered by replacing it with $\widehat{\Theta }^{-1}_{(k)}$.
\item The category $\mathcal{A}_{k^{(t)}}\langle \Theta \rangle$ will be replaced by $\mathcal{A}_{k^{(t+1)}}\langle \Theta \rangle$.
    \begin{itemize}
      \item Let $(W, \nu _{k^{(t)}})$ be a generic  object in $\mathcal{A}_{k^{(t)}}\langle \Theta \rangle$ where $\nu_{k^{(t)}}=(\nu_{1},\ldots, \nu_{k})$. Then a random object  $(\mu_{k}(W), \nu _{k^{(t+1)}})$ of   $\mathcal{A}_{k^{(t+1)}}\langle \Theta \rangle$  is given as follows

\begin{equation}\label{}
  \begin{cases} ( \Theta _{k}(W),\nu_{k^{(t+1)}}), \text {with}\\ \nu _{k^{(t+1)}}=(\nu_{1},\ldots,\nu_{k-1},\varepsilon_{k,D(\nu_{k})}\circ \nu_{k} ^{-1},\nu_{k+1},\ldots, \nu_{n}), & \text{ if } t \  \text {is even},\\
         (W, \nu_{k^{(t+1)}}), \text {with}\\\nu _{k^{(t+1)}}=(\nu_{1},\ldots,\nu_{k-1},\xi_{k, C(\nu_{k})}\circ\nu_{k} ^{-1},\nu_{k+1},\ldots, \nu_{n}), & \text{ if } t \  \text {is odd}.
    \end{cases}
    \end{equation}
      \item  Morphisms $Mor._{\mathcal{A}_{k^{(t+1)}}\langle \Theta \rangle}((W, \nu _{k^{(t)}}),(W', \nu ^{'R}_{k^{(t)}})$ are given by

\begin{equation}\label{}
 \begin{cases} \Theta(Mor._{\mathcal{A}_{k^{(t)}}\langle \Theta \rangle}((W, \nu ^{R}_{k^{(t)}}),(W', \nu ^{'R}_{k^{(t)}})), & \text{ if } t \  \text {is even},\\
         Mor._{ \mathcal{A}_{k^{(t)}}\langle \Theta \rangle}((W, \nu ^{R}_{k^{(t)}}),(W', \nu ^{'R}_{k^{(t)}})), & \text{ if } t \  \text {is odd}.
    \end{cases}
    \end{equation}
    \end{itemize}

 \item Let $\mathcal{S}=(\Theta, \xi, \mathcal{A}_{k^{(-t)}}\langle \Theta \rangle)$ be triple  obtained from  $\mathcal{P}$ by applying some  sequence of categorical mutations $\mu$ which contains  exactly $t$ copies of $\mu^{L}_{k}$  with $t\geq 2$.  The categorical left mutations rules for $\xi$ and $\Theta $ are the same as in the right mutations. The left categorical mutations of $\mathcal{A}_{k^{(-t)}}\langle \Theta \rangle$ with objects are the pairs $(W, \nu_{k^{(-t)}})$ such that
$\nu_{k^{(-t)}}=(\nu_{1},\ldots, \nu_{n})$ is replacing it with   $\mathcal{A}_{k^{(-t-1)}}\langle \Theta \rangle$ with  objects  given by

\begin{equation}\label{}
  \begin{cases} ( \Theta _{k}(W),\nu_{k^{(-t-1)}}), \text {with}\\ \nu_{k^{(-t-1)}}=(\nu_{1},\ldots,\nu_{k-1},\nu_{k} ^{-1}\circ \varepsilon_{k,C(\nu_{k})},\nu_{k+1},\ldots, \nu_{n}), & \text{ if } t \  \text {is even},\\
         (W, \nu_{k^{(-t-1)}}), \text {with}\\ \nu_{k^{(-t-1)}}=(\nu_{1},\ldots,\nu_{k-1},\nu^{-1}_{k}\circ \xi_{k,D(\nu_{k})},\nu_{k+1},\ldots, \nu_{n}), & \text{ if } t \  \text {is odd}.
    \end{cases}
    \end{equation}

The morphisms of $A_{k^{(-t-1)}}\langle \Theta \rangle$ are defined the same way as  the morphisms of $A_{k^{(t+1)}}\langle \Theta \rangle$ in (3.12).

   \end{enumerate}
   \end{enumerate}
Two categorical preseeds $\mathcal{P}$ and $\mathcal{S}$ are called \emph{mutationally equivalent} if they  can be obtained from each other by applying some  sequences of categorical mutations. In such case, every skew Laurent object of $\mathcal{P}$ is said to be mutationally equivalent to its associated skew Laurent object of $\mathcal{S}$.
\end{defn}

\begin{rem} \label{rem}

Let $\mathcal{P}$ be a hyperbolic categorical preseed. Then for every $i, j \in [1, n]$ we have

\begin{equation}
   \mu^{R}_{i}\mu^{R}_{j}\mathcal{(S)}=\mu^{R}_{j}\mu^{R}_{i}\mathcal{(S)}, \mu^{L}_{i}\mu^{L}_{j}\mathcal{(S)}=\mu^{L}_{j}\mu^{L}_{i}\mathcal{(S)} \ \text{and} \ \ \mu^{L}_{i}\mu^{R}_{j}\mathcal{(S)}=\mu^{R}_{j}\mu^{L}_{i}\mathcal{(S)}.
\end{equation}

\end{rem}

\begin{exam}[Categorical mutations of Weyl categorical preseed $\mathcal{W}$] In this example we give more precise forms for the objects of  the categories $\mathcal{A}_{k^{(t)}}\langle\Theta\rangle, k\in [1, n], t\in \mathbb{Z}$ of the Weyl categorical preseed $\mathcal{W}$ given in Example 3.8. For every $M\in Obj.(\mathcal{A})$, let  $\varepsilon_{i, M}: M \rightarrow M$ be an endomorphism of the identical functor of $\mathcal{A}$ induced by the element $\epsilon_{i}=\theta^{-1}_{i}(\xi_{i})\in D$ where $\varepsilon_{i, M}(m)=\epsilon_{i}(m), i\in [1,n]$. Let $(M, \overline{y})$ be some (generic) object in $ \mathcal{A}\langle \Theta \rangle$ and let $A_{k^{(t)}}\langle \Theta \rangle$ be a skew Laurent category obtained from the initial skew Laurent category $A\langle \Theta \rangle$ of  $\mathcal{W}$ by applying the sequence of categorical right mutations $\overbrace{\mu^{R}_{k}\cdots \mu^{R}_{k}}^{\text{t-times}}$.
The objects of $A_{k^{(t)}}\langle \Theta \rangle $ are of the form $(M, \overline{y}_{k^{(t)}})$,  $y _{k^{(t)}}=(\overline{y}_{1},\ldots, \overline{y}_{k-1},\overline{y}_{k^{t}},\overline{y}_{k+1},\ldots, \overline{y}_{n})$, where $\overline{y}_{k^{t}}$ is one of the following cases

\begin{equation}\label{}
  \begin{cases} \overline{y}_{k^{t}} \in Mor._{\mathcal{A}\langle \Theta \rangle}(\Theta(M), M);  \ \ \overline{y}_{k^{t}}(m)=\xi_{k}^{\frac{t}{2}}y_{k}\xi^{\frac{-t}{2}}_{k}(m), & \text{ if } t \  \text {is even},\\
          \overline{y}_{k^{t}} \in Mor._{\mathcal{A}\langle \Theta \rangle}(M, \Theta(M)); \ \ \overline{y}_{k^{t}}(m)=\xi_{k}^{\frac{t+1}{2}}y^{-1}_{k}\xi^{-\frac{(t-1)}{2}}_{k}(m), & \text{ if } t \  \text {is odd}.
    \end{cases}
    \end{equation}

Let $A_{k^{(-t)}}\langle \Theta \rangle$ be a skew Laurent category obtained from  $A\langle \Theta \rangle$ of $\mathcal{W}$ by applying the sequence of categorical left mutations $\overbrace{\mu^{L}_{k}\cdots \mu^{L}_{k}}^{\text{t-times}}$. The objects of $A_{k^{(-t)}}\langle \Theta \rangle $ are of the form $(M, \overline{y}_{k^{(-t)}})$, where $y _{k^{(t)}}=(\overline{y}_{1},\ldots, \overline{y}_{k-1},\overline{y}_{k^{-t}},\overline{y}_{k+1},\ldots, \overline{y}_{n})$, where $\overline{y}_{k^{t}}$ is one of the following cases

\begin{equation}\label{}
  \begin{cases} \overline{y}_{k^{-t}} \in Mor._{\mathcal{A}\langle \Theta \rangle}(\Theta(M), M);  \ \ \overline{y}_{k^{-t}}(m)=\xi_{k}^{\frac{-t}{2}}y_{k}\xi^{\frac{t}{2}}_{k}(m), & \text{ if } t \  \text {is even},\\
          \overline{y}_{k^{-t}} \in Mor._{\mathcal{A}\langle \Theta \rangle}(M, \Theta(M)); \ \ \overline{y}_{k^{-t}}(m)=\xi_{k}^{-\frac{t-1}{2}}y^{-1}_{k}\xi^{\frac{t+1}{2}}_{k}(m), & \text{ if } t \  \text {is odd}.
    \end{cases}
    \end{equation}

\end{exam}

\begin{lem} \label{lem} Let  $\mathcal{P}=(\Theta, \xi, \mathcal{A}\langle \Theta \rangle)$ be  categorical preseed of rank $n$ in $\mathcal{A}$. Then the following are true
\begin{enumerate}
\item [(1)] For every sequence of categorical right (respect to left) mutations   $\mu^{R}_{j_{1}}\cdots \mu^{R}_{j_{t}}$, we have  $\mu^{R}_{j_{1}}\cdots\mu^{R}_{j_{t}}(\mathcal{P})$ (respect to $\mu^{L}_{j_{1}}\cdots \mu^{L}_{j_{t}}(\mathcal{P}))$ is again a categorical preseed;
\item [(2)]
    For every $k\in [1,n]$, we have  $\mu^{R}_{k}\mu^{L}_{k}(\mathcal{P})=\mu^{L}_{k}\mu^{R}_{k}(\mathcal{P})=\mathcal{P}$;

\item [(3)] The categorical preseed $\mathcal{P}$ along with the categorical preseeds $\mu^{R}_{1}(\mathcal{P}),\ldots ,\mu^{R}_{n}(\mathcal{P})$ give
    rise to a hyperbolic category (respect to $\mu^{L}_{1}(\mathcal{P}),\ldots ,\mu^{L}_{n}(\mathcal{P}))$. This hyperbolic category will be  denoted by $\mathcal{H}^{R}(\mathcal{P})$ (respect to $\mathcal{H}^{L}(\mathcal{P})$).
\end{enumerate}

\end{lem}
\begin{proof}\begin{enumerate}
               \item [(1)]
In the following we prove Part (1) for categorical right mutations and for the categorical left mutations, the proof is  similar with the obvious changes. One can see that
Remark 3.10 reduces the proof of Part (1) into proving it only for the case $j_{i}=\ldots=j_{t}=k,$ for some $k\in [1,n]$. We start by proving Part (1) for sequences of mutations of lengths one and two, and  the  proof for  sequences of mutations with bigger lengths  is quite similar. Now, we show that $\mu^{R}_{k}(\mathcal{P})$ is again a categorical preseed. The main part is to prove  that $\mathcal{A}_{k^{(1)}}\langle \Theta \rangle$ satisfies the conditions of the skew Laurent category for the triple  $(\Theta^{-1}, \xi, \mathcal{A}_{k^{(1)}}\langle \Theta \rangle)$. One can see that if $(M, f)\in Obj.\mathcal{A}\langle \Theta \rangle$ then $\varepsilon_{k, \Theta(M)}\circ f^{-1}_{k}\in Mor_{\mathcal{A}}(\Theta^{-1}_{k}(\Theta_{k}(M)), \Theta_{k}(M))$. It  remains to show (3.4) for objects  and morphisms of $\mathcal{A}_{k^{(1)}}\langle \Theta \rangle$.
 Let $h$ be a morphism in $Mor_{\mathcal{A}\langle\Theta\rangle}((M, f), (M', f'))$. Then Diagram (3.4) for $\mathcal{P}$ tells us that $\Theta _{k}(h_{k})\circ f^{-1}_{k}=f ^{'-1}_{k}\circ h_{k}, k\in [1, n]$. Therefore  the following consecutive identities are satisfied
\begin{eqnarray}
 \nonumber \varepsilon_{k, \Theta(M')}\circ (\Theta _{k}(h_{k})\circ f ^{-1}_{k}) &=& \varepsilon_{k, \Theta(M')}\circ (f ^{'-1}_{k}\circ h_{k} )\\
   \Theta _{k}(h_{k})\circ (\varepsilon_{k, \Theta(M)}\circ f ^{-1}_{k})&=& (\varepsilon_{k,  \Theta(M')}\circ f ^{'-1}_{k})\circ  h_{k}.
\end{eqnarray}
Equation (3.17) above is due to  the fact that $\varepsilon_{k, -}$ being an endomorphism of the identical functor, thanks to (3.5) and the fact that $\xi_{k}$ is an endomorphism of the identical functor for every $k\in [1, n]$. Also, equation (3.17) says that the following diagram is commutative
\begin{equation}\label{}
   \nonumber \xymatrixcolsep{5pc}\xymatrix{\ar[d]^{\Theta_{k}(h)}\Theta_{k} (M)& \ar[l]_{\varepsilon_{k \Theta (M)}\circ f^{-1}_{k}}\ar[d]^{h} M\\
    \Theta_{k} (M') &\ar[l]^{\varepsilon_{k \Theta (M')}\circ f^{'-1}_{k}} M'}
\end{equation}
which is equivalent to the commutativity of the following diagram
\begin{equation}\label{}
  \nonumber  \xymatrixcolsep{5pc}\xymatrix{\ar[d]^{\Theta_{k}(h)}\Theta_{k} (M)& \ar[l]_{\varepsilon_{k, \Theta (M)}\circ f^{-1}_{k}}\ar[d]^{\Theta ^{-1}(\Theta_{k} (h))} \Theta_{k} ^{-1}(\Theta_{k} (M))\\
    \Theta_{k} (M') &\ar[l]^{\varepsilon_{k \Theta_{k} (M')}\circ f^{'-1}_{k}}\Theta_{k} ^{-1}(\Theta_{k} (M'))}
\end{equation}
which means
\begin{equation}\label{}
 \nonumber \Theta (h)\in Mor_{\mathcal{A}_{k}\langle\Theta\rangle}((\Theta(M), \varepsilon_{k, \Theta (M)}\circ f^{-1}_{k}), (\Theta(M'), \varepsilon_{k, \Theta (M')}\circ f^{'-1}_{k})).
\end{equation}

For right mutations of length two, altering $\mu^{R}_{k}(\mathcal{P})$ by mutation in $k$-direction, we get $\mu
_{k}^{R}\mu^{R}_{k}(\mathcal{P})=(\Theta, \xi, \mathcal{A}_{k^{(2)}}\langle\Theta\rangle$,
where objects of $\mathcal{A}_{k^{(2)}}\langle\Theta\rangle$ are given by the pairs $(M, \xi_{k,M}\circ (f_{k} \circ \varepsilon^{-1}_{k,\Theta _{k}(M)}))$. To prove that $(\Theta, \xi, \mathcal{A}_{k^{(2)}}\langle\Theta\rangle)$ is again a categorical preseed we need to prove (3.4) as follows
 \begin{eqnarray}
  \nonumber h\circ(\xi_{k,M}\circ (f_{k} \circ \varepsilon^{-1}_{k,\Theta _{k}(M')}))&=& \xi_{k, M'}\circ h\circ (f_{k} \circ \varepsilon^{-1}_{k,\Theta _{k}(M)}) \\
  \nonumber &=& \xi_{k, M'}\circ f'_{k}\circ \Theta _{k}(h) \circ \varepsilon^{-1}_{k,\Theta _{k}(M)} \\
 \nonumber &=&  (\xi_{k, M'}\circ f'_{k} \circ \varepsilon^{-1}_{k,\Theta _{k}(M')})\circ \Theta _{k}(h).
                  \end{eqnarray}
In the first equation we used that $\xi_{k}$ is an endomorphism of the identical functor of $\mathcal{A}$, in the second equation we used that $h\in hom_{A\langle \Theta \rangle}(M, M')$ and the last equation is due to $\varepsilon_{k}$ is also an endomorphism of the identical functor of $\mathcal{A}$.
 \item [(2)] To prove second part, we only need to  show  the equivalency of the two  categories   $\mu^{R}_{k}(\mu^{L}_{k}(\mathcal{A}\langle\Theta\rangle))$ and  $\mu^{L}_{k}(\mu^{R}_{k}(\mathcal{A}\langle\Theta\rangle))$. Actually, one  can see  that the category $\mathcal{A}\langle\Theta\rangle$ will be reproduced by applying $\mu^{R}_{k}\mu^{L}_{k}$ or $\mu^{L}_{k}\mu^{R}_{k}$, which is straightforward to prove.

    \item [(3)] To prove Part (3) for the categorical preseeds $\mathcal{P}, \mu^{R}_{1}(\mathcal{P}), \ldots,  \mu^{R}_{n}(\mathcal{P})$, we introduce the category  $\mathcal{H}^{R}(\mathcal{P})$ with objects are the triples $(\gamma, M, f)$, where  $\gamma=(\gamma_{1},\ldots,\gamma_{n})$ with $\gamma _{k}=\varepsilon_{k,\Theta _{k}(M)}\circ f^{-1}_{k}$, where $(M, f)$ is an object of $\mathcal{A}\langle\Theta\rangle$. The  morphisms of $\mathcal{H}^{R}(\mathcal{P})$ are given by

 \begin{equation}\label{}
 \nonumber Mor._{\mathcal{H}^{R}(\mathcal{P})}((\gamma, M, f),(\gamma', M', \eta'))=Mor._{\mathcal{A}\langle\Theta\rangle}((M, f), (M', f')).
 \end{equation}
 In the following we verify the conditions of the hyperbolic category. Starting with (3.2),  for every $k\in [1, n]$ we have

 \begin{equation}\label{}
   \nonumber \gamma_{k}\circ f_{k}=\varepsilon_{k, D(f_{k})}\circ f^{-1}_{k}\circ f_{k}=\varepsilon_{k, D(f_{k})},
 \end{equation}
and
\begin{equation}\label{}
   \nonumber f_{k}\circ \gamma_{k}=f_{k} \circ \xi_{k, D(f_{k})}\circ\eta^{-1}_{k}=f_{k} \circ f^{-1}_{k}\circ\xi_{k, C(f_{k})}=\xi_{k, C(f_{k})}.
 \end{equation}
 To prove (3.3), let $h\in Mor._{\mathcal{A}\langle\Theta\rangle}((M, f), (M', f'))$. Then $h\circ f_{k}=f'_{k}\circ \Theta_{k}(h)$ for every $k\in [1, n]$.
 Also
 \begin{eqnarray}
   \nonumber \Theta(h)\circ (\varepsilon_{k, D(f_{k})}\circ f_{k}^{-1}) &=& \Theta(h)\circ f_{k}^{-1}\circ \varepsilon_{k, C(f_{k})} \\
   \nonumber &=& f'^{-1}_{k}\circ h \circ \varepsilon_{k, C(f_{k})} \\
    \nonumber &=& f'^{-1}_{k}\circ \varepsilon_{k, C(h)}  \circ h \\
   \nonumber  &=& f'^{-1}_{k}\circ \varepsilon_{k, C(f')}  \circ h \\
   \nonumber  &=& (\varepsilon_{k, D(f')}\circ f'^{-1}_{k})\circ h.
 \end{eqnarray}

     \end{enumerate}
     The definition  of $\mathcal{H}^{L}(\mathcal{P})$ and the Proof of its being hyperbolic category is quite similar to the proof of $\mathcal{H}^{R}(\mathcal{P})$ with the obvious changes.
\end{proof}

\begin{cor} For every two mutationally equivalent categorical preseeds $\mathcal{P}$ and $\mathcal{S}$ there exist a unique  sequence of  categorical mutations $\mu$  such that $\mathcal{S}=\mu(\mathcal{P})$. Where for every $i\in [1, n]$, $\mu_{i}$ appears in  $\mu$ only once and   in  one of the forms $(\mu^{L}_{i})^{n_{i}}$ or $(\mu^{R}_{i})^{n_{i}}$ for some non-negative number $n_{i}$
\end{cor}
\begin{proof} We start by proving the following statement: Every mixed sequence  of  categorical right and left mutations $\mu=\mu_{i_{1}}\cdots \mu_{i_{k}}$  can be reduced into a smaller or equal length sequence of categorical mutations of the form $\mu^{R}\mu^{L}$ or equivalently $\mu^{L}\mu^{R}$ where
\begin{equation}\label{}
 \mu^{R}=(\mu^{R}_{i_{t_{1}}})^{n_{1}}\cdots(\mu^{R}_{i_{t_{q}}})^{n_{q}} \ \ \text{and} \ \  \mu^{L}=(\mu^{L}_{i_{d_{1}}})^{n'_{1}}\cdots(\mu^{L}_{i_{d_{q'}}})^{n'_{q'}},
\end{equation}
with $i_{t_{j}}, i_{d_{j}}, q$ and $q'\in [1, n], n_{j}, n'_{j}\in \mathbb{N}$ for  $j\in [1, max(q, q')]$ such that $\{i_{t_{1}}, \cdots, i_{t_{q}} \} \cap \{i_{d_{1}}, \cdots, i_{d_{q'}}\}={\O} $. We use (3.14)  along with Part (2) of Lemma 3.12  to rewrite every sequence of mixed categorical right and left  mutations as $\mu^{R}\mu^{L}$ or $\mu^{L}\mu^{R}$ where $\mu^{R}$ is a sequence of only categorical right mutations and $\mu^{L}$ is a sequence of only categorical left mutations as given in (3.18).
One can see that the statement of the corollary is true for any two categorical preseeds of rank 1.  And since  every sequence of categorical mutations $\mu$ can be written as the product $\mu^{R}\mu^{L}$ such that $\mu^{R}$ and $\mu^{L}$ are as given in (3.18), where each sequence of categorical mutations is written as a product of independent sequences of categorical mutations each one is in one direction. Where the independence of these sequences is guaranteed from the definition of categorical mutations, Definition 3.9. Each one of theses  sequences connects  one component, say $\gamma_{k}$ of a generic skew Laurent object $(M, (\gamma_{1}, \ldots, \gamma_{n}))$ in  $\mathcal{P}$  to same-index component of the corresponding  skew Laurent object in $\mathcal{S}$ which can be seen as a sequence between two rank 1 categorical preseeds. So each subsequence $(\mu^{R}_{i_{t_{1}}})^{n_{i}}$ and $(\mu^{L}_{i_{t_{1}}})^{n_{i}}$ is unique thus $\mu$ must be unique.
\end{proof}

\subsection{Clustered hyperbolic categories}

Let $\mathcal{P}$ and $\mathcal{S}$ be two mutationally equivalent categorical preseeds, that is $\mathcal{S}=\mu^{R}(\mu^{L}(\mathcal{P}))$ or $\mathcal{S}=\mu^{L}(\mu^{R}(\mathcal{P}))$ where $\mu^{R}$ and $\mu^{L}$ are as given in (3.18). Then we  define the hyperbolic categories $\mathcal{H}^{R}(\mathcal{S})$ and $\mathcal{H}^{L}(\mathcal{S})$ of $\mathcal{S}$ by the same way as it was defined in the Proof of Part (3) of Lemma 3.12. The category $\mathcal{H}^{R}(\mathcal{S})$  (respect to $\mathcal{H}^{L}(\mathcal{S})$) will be called the right (respect to left) hyperbolic category of $\mathcal{S}$. Here is a more precise definition of $\mathcal{H}^{R}(\mathcal{S})$ and the definition of $\mathcal{H}^{L}(\mathcal{S})$ is similar with the obvious changes. Objects of $\mathcal{H}^{R}(\mathcal{S})$  are given as follows, if $(M, (g_{1},\ldots , g_{n}))$ is a typical object in the skew Laurent category $\mathcal{A}_{\mathcal{S}}\langle \Theta \rangle$ of $\mathcal{S}$  then $(\mu^{R}_{1}(g_{1}), \ldots, \mu^{R}_{n}(g_{n}),M , (g_{1},\ldots , g_{n}))$ is an object in $\mathcal{H}^{R}(\mathcal{S})$ (respect to $(\mu^{L}_{1}(g_{1}), \ldots, \mu^{L}_{n}(g_{n}),M , (g_{1},\ldots , g_{n}))$ is an object in $\mathcal{H}^{L}(\mathcal{S})$). The morphisms are given by
\begin{equation}\label{}
Hom_{\mathcal{H}^{R}(\mathcal{S})}((g', M, g), (h', M', h))=Hom_{\mathcal{A}_{\mathcal{S}}\langle \Theta \rangle}((M, g), (M', h)).
\end{equation}

\begin{defns} \begin{enumerate}
 \item The triple $(\gamma, M, \eta)$ is called a \emph{hyperbolic object} of the categorical preseed $\mathcal{P}$ if it is an object in the following set
                  \begin{equation}\label{}
              \nonumber H(\mathcal{S})=Obj.\mathcal{H}^{R}(\mathcal{S})\cup Obj.\mathcal{H}^{L}(\mathcal{S}),
                  \end{equation}
                  for some categorical preseed  $\mathcal{S}$ that is mutationally equivalent to $\mathcal{P}$.
 \item A \emph{mutation hyperbolic category} $\mathcal{H}(\mathcal{P})$ of a categorical preseed $\mathcal{P}$ is a category with set of  objects  consists of the hyperbolic objects of $\mathcal{P}$. The  morphisms are given by: $h=(h_{1},\ldots, h_{n}) \in  Mor._{\mathcal{H}(\mathcal{P})}((f', M, f), (g', W, g) )$  where   $h_{i} \in Mor._{\mathcal{A}}(M, W), i\in [1, n]$  satisfies the commutativity of   the diagram

 \begin{equation}\label{}
   \xymatrix{
 M \ar[d]^{h_{i}} \ar[r]^{f'_{i}} & \ar[d]^{\Theta_{i}(h_{i})}\Theta_{i} (M)\ar[r]^{f_{i}}& \ar[d]^{h_{i}} M\\
  W \ar[r]^{g'_{i}}  &  \Theta_{i} (W)\ar[r]^{g_{i}} & W.}
\end{equation}

\item A \emph{clustered hyperbolic category} $\mathfrak{C}$ in a categorical preseed $\mathcal{P}$ is defined to be any full subcategory of the mutation hyperbolic category $\mathcal{H(P)}$  where for every object  $(f', M, f)$   in Obj.$\mathfrak{C}$ the pair $(M, f)$ is an object in one and only one skew Laurent category $\mathcal{A}\langle \Theta \rangle$ for some categorical pressed $\mathcal{S}$ that is   mutationally equivalent to $\mathcal{P}$.

  \end{enumerate}
\end{defns}
 In the following we will omit the word hyperbolic from the expression hyperbolic mutation category.
\begin{rem} With the setting of Definitions 3.14 the categories $\mathcal{H}^{R}(\mathcal{S})$ and $\mathcal{H}^{L}(\mathcal{S})$ are full subcategories of $\mathcal{H}(\mathcal{P})$.
\end{rem}
\begin{proof} By the definition the sets of objects of the categories $\mathcal{H}^{R}(\mathcal{S})$ and $\mathcal{H}^{L}(\mathcal{S})$, they are subsets of the set of objects of $\mathcal{H}(\mathcal{P})$. For the morphisms, obviously  $Mor._{\mathcal{H}(\mathcal{P})}((f', M,f),(g', W, g)) \subseteq Mor._{\mathcal{H}^{R}(\mathcal{S})}((f', M,f),(g', W, g))$ (respect to $\mathcal{H}^{L}(\mathcal{S})$). For the other direction,  if $h$ satisfies (3.19) then it make the right half of the diagram in (3.20) commutative. Now since, $f'=\mu_{k}^{R}(f)$ and $g'=\mu_{k}^{R}(g)$ (respect to $\mu^{L}_{k}(-)$), hence for ever $k\in [1, n]$ we have

\begin{equation}\label{}
   \nonumber \xi_{k,\Theta_{k}(W)}g_{k}^{-1}h_{k}=\xi_{k,\Theta(W)}\Theta_{k}(h_{k})f^{-1}_{k}=\Theta_{k}(h_{k})\xi_{k,\Theta(M)}f^{-1}_{k},
\end{equation}
where the first equation is by using the commutativity of the right half of the diagram (3.20) and the second equation is due to the fact that $\xi_{k}$ satisfies (3.5). The case of left categorical mutation is quite similar and in case of using $\varepsilon_{k}$ instead of $\xi_{k}$ there will be no major changes.  which finishes the proof.
\end{proof}
One of the main goals of this article is providing a general technique to identify a full subcategory  of every mutation  category that is a clustered hyperbolic category. For this sake,  we introduce \emph{cluster classes} and  \emph{hyperbolic class}.

\begin{defn}
\begin{itemize}
Let $(M, \eta)$ be an object in a skew Laurent category $\mathcal{A}\langle \Theta \rangle$. Then $(M, \eta)$ forms  the following two sets which we call \emph{cluster classes}
\begin{enumerate}
  \item [(a)] $[(M, \eta)]_{0}$ consists of all hyperbolic objects  $(\gamma', M', \eta')$ where $(M', \eta')$ can be obtained from  $(M, \eta)$ by applying a sequence of categorical (right or left) mutations.
  \item [(b)] $[(M, \eta)]_{1}$ consists of all hyperbolic objects  $(\gamma', M', \eta')$ where $(M', \eta')$ can be obtained from  $\mu^{L}_{k}(M, \eta)$ by applying a sequence of categorical right mutations or from $\mu^{R}_{k}(M, \eta)$ by applying a sequence of categorical left mutations.
 \end{enumerate}
 We also introduce the set   $\overline{(M, \eta)}=[(M, \eta)]_{0}\cup[(M, \eta)]_{1}$.
\item  The \emph{ hyperbolic class} of $(\gamma, M,\eta)$ at the categorical preseed $\mathcal{S}$ is given by
 \begin{equation}\label{}
 \nonumber  \langle(M, \eta)\rangle(\mathcal{S})=\{(h', M, h)\in \overline{(M, \eta)};  (M, h) \ \text{is an object in} \ \mathcal{A}\langle \Theta \rangle \ \text{of} \ \mathcal{S}\},
 \end{equation}
  or in other words
  \begin{equation}\label{}
  \langle(M, \eta)\rangle(\mathcal{S})=\overline{(M, \eta)}\cap(Obj.\mathcal{H}^{R}(\mathcal{S}) \cup Obj.\mathcal{H}^{L}(\mathcal{S}))
\end{equation}
\end{itemize}

\end{defn}

\begin{rem}

Let $\mathcal{P}$ be a categorical preseed in $\mathcal{A}$.  Then the following three conditions are equivalent

\begin{enumerate}
  \item  A mutation hyperbolic category $\mathcal{H}(\mathcal{P})$ is a clustered hyperbolic category.
  \item For every object $M$ in $\mathcal{A}$,  each hyperbolic class $\langle(M, \eta)\rangle(\mathcal{S})$  contains exactly two elements for every categorical preseed $\mathcal{S}$ that is mutationally equivalent $\mathcal{P}$.
  \item The skew Laurent category $\mathcal{A}\langle \Theta \rangle$ of the initial categorical preseed $\mathcal{P}$ contains no mutationally equivalent objects.
\end{enumerate}

In the following we introduce a presentation for cluster classes. All the definitions and proofs are written for a cluster class $[(M, \eta)]\in \{[(M, \eta)]_{0}, [(M, \eta)]_{1}\}$.

\end{rem}
\begin{defns} \begin{enumerate} (Zigzag presentations of cluster classes)
\item [(1)] Let $h=(\gamma, M, \eta)$ and $h'=(\gamma', M', \eta')$ be two hyperbolic objects. We say that $h$ divides $h'$ at  $k\in [1, n]$, in symbols it shall be written as  $h|_{k}h'$, if  either of $\eta_{k}$ or $\eta^{-1}_{k}$ divides $\eta'_{k}$; that is  $\eta'_{k}$ can be written as a composition of morphisms in the form $g\eta_{k} h$ or $g\eta^{-1}_{k} h$, for some morphisms $g$ and $h$. In particular if $h|_{k}h'$ and $h'=\mu_{k}^{L}(h)$ or $h'=\mu_{k}^{R}(h)$ we simply use the following diagrams respectively

  \begin{equation}\label{}
    \nonumber 
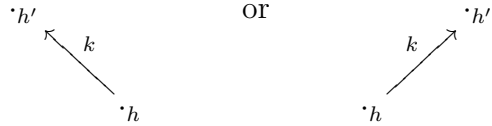
\xymatrix{\cdot_{h'}  \\
        &\cdot_{h} \ar[ul]_{k}} \ \ \ \ \ \ \ \ \  \text{or} \ \ \ \ \  \ \ \   \xymatrix{  & \cdot_{h'}  \\
        \cdot_{h} \ar[ur]^{k} }
     \end{equation}

\item [(2)] \begin{itemize}
              \item
A morphism $\alpha$ is called an \emph{initial morphism} at $k \in [1, n]$  if for every two morphisms  $\alpha'$ and $\alpha''$   that can be obtained from  $\alpha$ by applying a single left and  right categorical mutations at $k$ respectively, then  we have  $\alpha |_{k}\alpha'$ and $\alpha |_{k}\alpha''$.  A hyperbolic object $h_{0}=(\gamma, M_{0}, \eta)$ is called an \emph{initial object} of the cluster class $[(M, (\eta_{1}, \ldots, \eta_{n})]$ if for every $k$ in $[1, n]$, the morphism $\eta_{k}$ is an initial morphism  at $k$, which can be presented by

      \begin{equation}\label{}
    \nonumber    
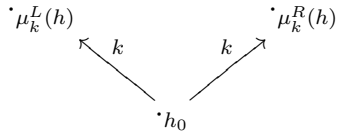
\xymatrix{\cdot_{\mu^{L}_{k}(h)}&  & \cdot_{\mu^{R}_{k}(h)}  \\
       & \cdot_{h_{0}} \ar[ur]^{k} \ar[ul]_{k}}
     \end{equation}
     \begin{center}
for  every $k\in [1, n]$.
\end{center}

 \item Let $[(M, \eta)]$ be a cluster class with initial  objects $h_{1}, \ldots, h_{n}, n>1$ such that the biggest number of non initial objects between $h_{i}$ and $h_{i+1}$ for every $i \in [1, n-1]$  is $m$. Then we say that $[(M, \eta)]$ has a \emph{zigzag presentation} of length $n$ and height $m$ and  in case of $n=1$, the height would be a zero. The initial objects $h_{1}$ and $h_{n}$ are called the \emph{left} and the \emph{ right initial objects} of $[(M, \eta)]$ respectively.

Consider the following   zigzag presentation of length $n$ and height one

 \begin{equation}\label{}
    \nonumber    
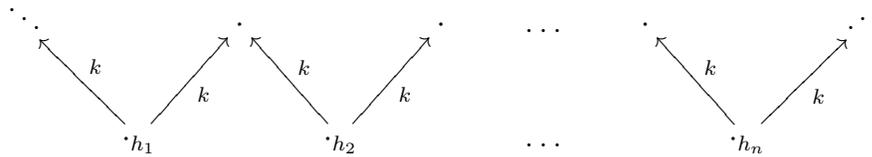
\xymatrix{  \ddots &&\cdot &&\cdot & \ldots &\cdot &&\iddots\\
    &\cdot_{h_{1}}\ar[ul]_{k}\ar[ur]_{k}&&\cdot_{h_{2}}\ar[ul]_{k}\ar[ur]_{k}&& \ldots &  &\cdot_{h_{n}} \ar[ul]_{k}\ar[ur]_{k} }
     \end{equation}

              \item A non-zero morphism $\eta$ is called a \emph{naked morphism} at $k\in [1, n]$ if it is not divisible by any of the morphisms  $\varepsilon^{\pm 1}_{k, M}, \xi^{\pm 1}_{k, M}, \varepsilon^{\pm 1}_{k, \Theta(M)}$ or $\xi^{\pm 1}_{k, \Theta(M)}$. A skew Laurent object $h=(M, (\eta_{1}, \ldots, \eta_{n}))$ is called a naked object if  $\eta_{k}$ is a naked morphism for every $k\in [1, n]$.
            \end{itemize}

     \end{enumerate}

\end{defns}

\begin{exams}  Consider the Weyl categorical preseed given in Example 3.11 using Example 3.8. Recall that $\mathcal{A}$ is the category of all modules over $D_{n}\{\theta, \xi\}$  considering them as $D$-modules. Let $M$ be an object in $\mathcal{A}$.
\begin{enumerate}
                \item The object $(M, \overline{y})$ is a naked object in $\mathcal{A}\langle\Theta\rangle$ and the $k^{\text{th}}$ branch of the zigzag presentation of $[(M, \overline{y})]_{0}$ is as follows

  \begin{equation}\label{}
    \nonumber    
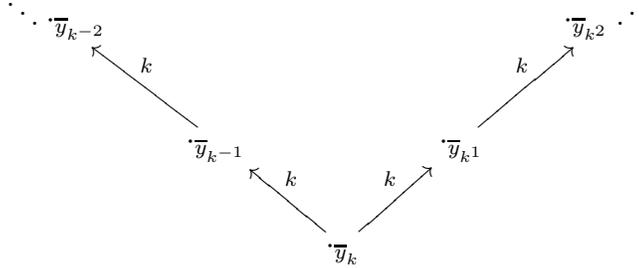
\xymatrix{\ddots \cdot_{\overline{y}_{k^{-2}}} &&&& \cdot_{\overline{y}_{k^{2}}} \iddots \\
    &\cdot_{\overline{y}_{k^{-1}}}\ar[ul]_{k}&& \cdot_{\overline{y}_{k^{1}}} \ar[ur]^{k}& \\
       & & \cdot_{\overline{y}_{k}} \ar[ur]^{k} \ar[ul]_{k}}
     \end{equation}

                \item Consider the following skew Laurent  object $(M, \xi_{M}^{-1} \eta)$ where the morphism $\eta$ is a naked morphism. The cluster class  $[(M, \xi_{M}^{-1} \eta)]_{0}$ has  the following  zigzag presentation of length two and height one
 \begin{equation}\label{}
    \nonumber    
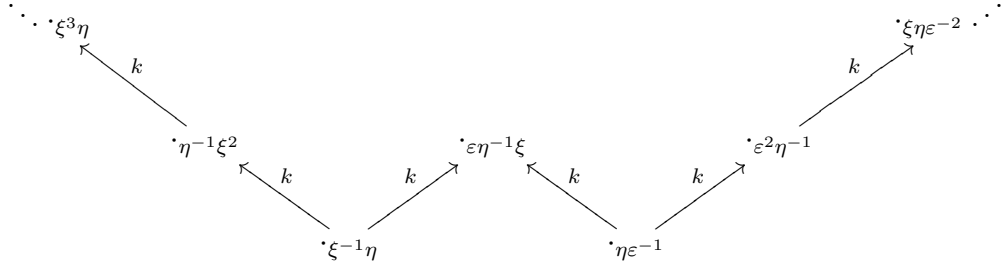
\xymatrix{\ddots \cdot_{\xi^{3}\eta} &&&&& & \cdot_{\xi\eta\varepsilon^{-2} }\iddots  \\
  &\cdot_{\eta^{-1}\xi^{2}}\ar[ul]_{k} & &\cdot_{\varepsilon \eta^{-1}\xi} & & \cdot_{\varepsilon^{2}\eta^{-1}\ar[ur]^{k} }  \\
     & &\cdot_{\xi^{-1} \eta} \ar[ur]^{k} \ar[ul]_{k} && \cdot_{\eta \varepsilon^{-1}} \ar[ur]^{k} \ar[ul]_{k}}
     \end{equation}
The left and right initial hyperbolic objects are $(M, \xi_{M}^{-1} \eta)$ and $(M,  \eta\varepsilon_{\Theta(M)}^{-1})$ respectively.

 \item The zigzag presentation of the cluster class  $[(M, \xi_{M}^{-1} \eta)]_{1}$ with $\eta$  naked morphism is of length one and hight zero, as follows
  \begin{equation}\label{}
    \nonumber    
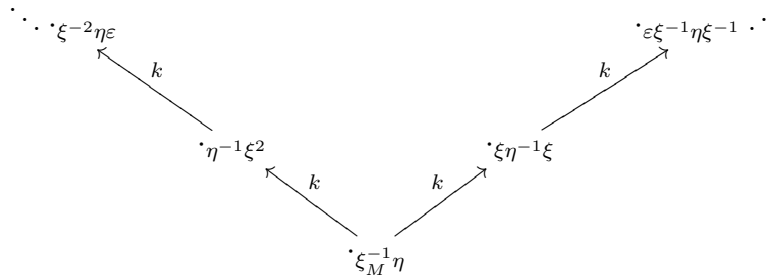
\xymatrix{\ddots \cdot_{\xi^{-2}\eta\varepsilon} &&&& \cdot_{\varepsilon\xi^{-1} \eta\xi^{-1}} \iddots \\
    &\cdot_{\eta^{-1}\xi^{2}}\ar[ul]_{k}&& \cdot_{\xi\eta^{-1}\xi} \ar[ur]^{k}& \\
       & & \cdot_{\xi_{M}^{-1} \eta} \ar[ur]^{k} \ar[ul]_{k}}
     \end{equation}
\item An object $M$  is called \emph{Verma object} in $\mathcal{A}$ if $\xi_{k}M=0$ for all $k\in [1, n]$, for more details about Verma modules of generalized Weyl algebras see [18]. Therefore, if $M$ is a Verma object, then the actions of the morphisms  $\varepsilon^{\pm 1}_{k, M}, \xi^{\pm 1}_{k, M}, \varepsilon^{\pm 1}_{k, \Theta(M)}$ and $\xi^{\pm 1}_{k, \Theta(M)}$ on $M$ are by multiplying by one of the following constants $\{-1,0 ,1\}$ which means   $(M, \eta)$ is  actually a naked object in $\mathcal{A}\langle\Theta\rangle$ for any non-zero morphism $\eta$.
\end{enumerate}

\end{exams}

\begin{prop} Let $h$ and $h'$ be two hyperbolic objects in the same cluster class. If $h|_{k}h'$ then there is no initial hyperbolic objects in between them. In other words if $\mu^{R}_{k^{t}}(h)=h'$ (respect to $\mu^{L}_{k^{t}}(h)=h'$) then non of the hyperbolic objects $\mu^{R}_{k^{d}}(h), 1\leq d<t$  (respect to $t\leq  d \leq -1$) is  initial object.
\end{prop}
\begin{proof} The proof is written for the case $\mu^{R}_{k^{t}}(h)=h'$ and the proof of  categorical left mutation case is similar.  We start with the case when $h$ is not a naked morphism and we will prove it by mathematical induction on the length of $d$. If $h$ is not naked morphism then it must be  divisible by some of the morphisms $\varepsilon^{\pm 1}_{k, M}, \xi^{\pm 1}_{k, M}, \varepsilon^{\pm 1}_{k, \Theta(M)}$ or $\xi^{\pm 1}_{k, \Theta(M)}$. If $d=1$ then no thing to prove. Assume that $d=2$ and let $h''=\mu^{R}_{k}(h)$. Then if $h''$ is an initial object then it divides $h$ at $k$ which means that the categorical mutation $\mu^{R}_{k}$ must have canceled one of the morphisms $\varepsilon^{\pm 1}_{k, M}, \xi^{\pm 1}_{k, M}, \varepsilon^{\pm 1}_{k, \Theta(M)}$ or $\xi^{\pm 1}_{k, \Theta(M)}$ that appears in the expression of $h''$, assume, without loss of generality, that the canceled morphism is one of $\xi^{\pm 1}_{k, \Theta(M)}$ or $\xi^{\pm 1}_{k, M}$. By definition of categorical  mutation  it alternates between multiplying by the morphisms $\xi$ and $\varepsilon$, then applying the categorical right mutation again on $h''$ we will obtain a morphism with one of the expressions $\varepsilon^{\pm 1}_{k, M} (h'')^{\pm1}$ or $\varepsilon^{\pm 1}_{k, \Theta(M)} (h'')^{\pm1}$.  Then the multiplicity of either of the  elements $\xi^{\pm 1}_{k, \Theta(M)}$ or $\xi^{\pm 1}_{k, M}$ in the expression of the $\mu^{R}_{k}(h'')$, which equals $h'$,  is less than their multiplicity in the expression of $h$  which contradicts with the fact that $h|_{k}h'$. Suppose that the statement is true for $d=q$, i.e., non of the hyperbolic objects $\mu^{R}_{k^{t}}(h)$ is initial object for $1<t \leq q $ then using the same argument we used in the induction step at $d=2$, we must have  $\mu^{R}_{k^{t}}(h)|_{k}\mu^{R}_{k^{t+1}}(h)$,  $1<t \leq q $, $q>2$. Then applying categorical right mutations accumulates the morphisms $\varepsilon^{\pm 1}_{k, M}, \xi^{\pm 1}_{k, M}, \varepsilon^{\pm 1}_{k, \Theta(M)}$ and $\xi^{\pm 1}_{k, \Theta(M)}$ around the naked morphism of $h$ and never cancel any of them,  which means $\mu^{R}_{k^{q}}(h)$ has one of the forms $\xi^{q} h \varepsilon^{-(q-1)}$ or  $\varepsilon^{q} h \xi^{-(q-1)}$. Thus applying categorical right mutation one more time will not produce any smaller form that can divide either of $\xi^{q} h \varepsilon^{-(q-1)}$ or  $\varepsilon^{q} h \xi^{-(q-1)}$ which means it is not possible to produce an initial object using $\mu^{R}_{k^{q+j}}$ for $j>1$. Which finishes the proof in the non-naked case.
If $h$ is a naked morphism then  every categorical mutations step applied to $h$ will result in multiplying $h^{-1}$ or $h$ by one of the morphisms $\varepsilon^{\pm 1}_{k, M}, \xi^{\pm 1}_{k, M}, \varepsilon^{\pm 1}_{k, \Theta(M)}$ or $\xi^{\pm 1}_{k, \Theta(M)}$ which gives us the following relations $\mu_{k^{t}}(h)|_{k}\mu_{k^{t+1}}(h), t \in [1, d-1]$, therefore non of $\mu_{k^{t}}(h), t \in [1, d-1]$ could be an initial object.

\end{proof}

\begin{cor} If  $h$ is a naked object then  the zigzag presentation of $[(M, h)]$ is of length one and height zero, i.e., it  is an open-up cone with vertex at $h$.
\end{cor}
\begin{proof} Since $h$ is a naked object then $h|_{k} \mu^{R}_{k^{t}}(h)$ and $h|_{k} \mu^{L}_{k^{t}}(h)$ for every $k\in [1, n]$ and $t \in \mathbb{Z}_{\geq 0}$. Hence from Proposition 3.20, $h$ must be the only initial object in $[(M, h)]$.
\end{proof}
\begin{lem} Every  cluster class $[(M, \eta)]$ has a zigzag presentation of a finite length and of a hight one. In particular we have

\begin{enumerate}
  \item [(1)] Initial objects are always exist;
 \item [(2)] The hight of any zigzag presentation is at most one;
 \item [(3)] Number of initial hyperbolic objects is finite.
\end{enumerate}
\end{lem}
\begin{proof} The proof is divided into three steps

\begin{enumerate}
\item [(1)] We will prove this part by showing that, every hyperbolic object is divisible by some initial object. Let $h=((h'_{1}, \ldots, h'_{n}), W, (h_{1}, \ldots, h_{n}))$ be a hyperbolic object in $[(M, \eta)]$,  where $W \in \{M, \Theta(M)\}$. If $h$ is not an initial hyperbolic object of $[(M, \eta)]$ then  there exists a hyperbolic object $g_{1}^k=((g'_{1,1}, \ldots, g'_{1, n}), W_{1}, (g_{1,1}, \ldots, g_{1, n}))$  such that $g^{k}_{1}|_{k}h$ for some  $k\in [1, n]$. Without loss of generality we can assume that $h=\mu^{R}_{k}(g_{1}^{k})$ (the proof of the case $h=\mu^{L}_{k}(g_{1}^{k})$ is similar).  Thus $h_{k}$ can be written in the form $h_{k}=\alpha_{1, k}g_{1, k}^{\pm 1}$ where $\alpha_{1, k}\in \{\varepsilon^{\pm 1}_{k, M}, \xi^{\pm 1}_{k, M}, \varepsilon^{\pm 1}_{k, \Theta(M)}, \xi^{\pm 1}_{k, \Theta(M)}\}$. If $g_{1}^k$ is an initial hyperbolic object then the proof of existence is finished. Assume that $g_{1}^k$ is not initial hyperbolic object. Then there is   $g_{2}^k=((g'_{2,1}, \ldots, g'_{2, n}), W_{2}, (g_{2,1}, \ldots, g_{2, n}))$  such that $g^{k}_{2}|_{k}g_{1}^k$, so $g_{1, k}=\alpha_{2, k}g_{2, k}^{\pm 1}$ for some  $\alpha_{2, k}\in \{\varepsilon^{\pm 1}_{k, M}, \xi^{\pm 1}_{k, M}, \varepsilon^{\pm 1}_{k, \Theta(M)}, \xi^{\pm 1}_{k, \Theta(M)}\}$. Again, without loss of generality,  we can assume that  $h=\mu^{R}_{k^{2}}(g_{2,k})$. We can continue in this process for only finite number of steps because $h$ must be at most a finite product, so the length of this process is proportional to the multiplicity of  $\xi_{k}^{\pm 1}$ and $\varepsilon_{k}^{\pm 1}$ dividing $h_{k}$. Therefore we obtain a sequence of morphisms $g_{1, k}, \ldots, g_{d_{k}, k}$ such that $g_{j+1, k}|_{k}g_{j, k}, j\in \{1, \ldots, d_{k}\}$. Since $k$ was a random element from $[1,n]$ so we can apply the same process on each morphism $h_{1}, \ldots, h_{n}$, we will end up to the hyperbolic object  $h_{0}=((g'_{d_{1}, 1}, \ldots, g'_{d_{n}, n}), M, (g_{d_{1}, 1}, \ldots, g_{d_{n}, n}))$.
    To show that $h_{0}$ is initial object, let $l=((l'_{1}, \ldots, l'_{n}), W, (l_{1}, \ldots, l_{n}))$ be a hyperbolic object such that $l=\mu^{L}_{k}(h_{0})$ for some $k\in [1, n]$. If $h_{0}$ is a naked object then we must have $h_{0}|_{k}l$. If $h_{0}$ is not a naked object and  $l|_{k}h_{0}$, so since $\mu^{R}_{k}(l)=h_{0}$ then $l$ must be one of $g_{1, k}, \ldots, g_{d_{k}, k}$  which contradicts with the fact  that $h_{0}$ is not divisible by any hyperbolic object that produces it by a single categorical right mutation.\\
    One can see, in the case of $h=\mu^{R}_{k}(g_{1}^{k})$ (respect to $h=\mu^{L}_{k}(g_{1}^{k})$) then $h$ will appear on the right (respect to left) of $h_{0}$ on the zigzag presentation.

\item [(2)] Suppose that $[(M, h)]$ is cluster class with zigzag presentation of hight bigger than one, without loss of generality assume that the hight is two. Let $h, h_{1}, h_{2}$ and $h_{3}$ be connected hyperbolic objects such that both of $h$ and $h_{3}$ are initial objects. Then the part of the zigzag presentation that contains $h, h_{1}, h_{2}$ and $h_{3}$ would look as follows

     \begin{equation}\label{}
    \nonumber    
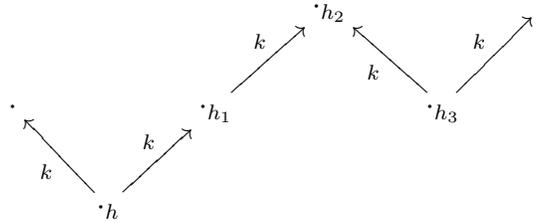
\xymatrix{ &&&\cdot_{h_{2}}&& \cdot \\
  \cdot & &\cdot_{h_{1}}\ar[ur]^{k} & & \cdot_{h_{3}}\ar[ul]^{k}\ar[ur]^{k}&&    \\
      &\cdot_{h} \ar[ur]^{k} \ar[ul]^{k} &&}
     \end{equation}

     Since $h|_{k}h_{1}$ and $h_{1}=\mu^{R}_{k}(h)$ then the multiplicity of one of the morphisms $\tilde{\xi}=\{\varepsilon^{\pm 1}_{k, M}, \xi^{\pm 1}_{k, M}, \varepsilon^{\pm 1}_{k, \Theta(M)}, \xi^{\pm 1}_{k, \Theta(M)}\}$ has increased by one. Now applying $\mu^{R}_{k}$ one more time will result in increasing of the multiplicity of another morphism from the set $\tilde{\xi}$. Hence by definition of categorical right mutation we must have $h_{2}$ has one of the following forms $\xi_{k, -}h \varepsilon_{k, -}^{-1}$ or $\varepsilon_{k, -} h \xi^{-1}_{k, -}$ which means $h_{3}$ has one of the expressions  $\varepsilon_{k, -}^{2} h^{-1} \xi^{-1}_{k, -}$ or $\xi^{2}_{k, -}h^{-1} \varepsilon_{k, -}^{-1} $. In both $h_{2}$ and $h_{3}$ the multiplicities of the morphisms $\{\varepsilon^{\pm 1}_{k, M}, \xi^{\pm 1}_{k, M}, \varepsilon^{\pm 1}_{k, \Theta(M)}, \xi^{\pm 1}_{k, \Theta(M)}\}$ can not be reduced any more. Therefore $h_{3}$ does not divide $h_{2}$ at $k$ which means $h_{3}$ can not be initial object.

\item [(3)]
In the proof of this part we will also  explain how the non naked initial morphisms occur. Let $h$ be non naked initial object. Then $h$ divides both of $h'=\mu^{L}_{k}(h)$ and $h''=\mu^{R}_{k}(h)$ which means that applying $\mu^{R}_{k}$ on $h'$ reduces the multiplicity of one of the morphisms, say $\alpha \in \{\varepsilon^{\pm 1}_{k, M}, \xi^{\pm 1}_{k, M}, \varepsilon^{\pm 1}_{k, \Theta(M)}, \xi^{\pm 1}_{k, \Theta(M)}\}$. Then applying $\mu^{R}_{k}$ on $h$  increases the multiplicity of another  morphisms say   $\beta \in \{\varepsilon^{\pm 1}_{k, M}, \xi^{\pm 1}_{k, M}, \varepsilon^{\pm 1}_{k, \Theta(M)}, \xi^{\pm 1}_{k, \Theta(M)}\}\setminus \{\alpha\}$. Applying $\mu^{R}_{k}$ again on $h''$ will either increase or decrease the multiplicity of the morphism $\alpha$. If it is increasing then $h$ is a right corner, i.e., $h$ is a right end initial object and no any other initial objects on its right side, thanks to Part (2) of this lemma. If it is decreasing, it would be decreasing the multiplicity of $\alpha$ then it creates another initial object. Continue in applying $\mu^{R}_{k}$ will produce number of initial objects that is proportionate to the multiplicity of $\alpha$ which is finite.  Then number of initial objects to the right of $h$ is finite. In a quite similar way we can prove that the number of initial objects to the left of $h$ is also finite. Which means that the total number of the initial objects must be finite.

\end{enumerate}
\end{proof}

\begin{cor} Every hyperbolic object is connected to  unique right end  and  left end initial objects. The Proof of Lemma 3.22 provides  a precise technique to identify the unique right (respect to left) end initial object of every cluster class.

\end{cor}
\begin{proof} Let $h$ be a non initial object of some cluster class. Then following the proof of Part (1) of Lemma 3.22 leads us to an initial hyperbolic object (lies to left or right  of $h$ in the zigzag presentation), let's call it $h'$. If there are some hyperbolic objects  sandwiched between $h$ and $h'$, such that there is a natural number $m$  so that $h'|_{k}\mu^{R}_{k}(h')|_{k}\mu^{R}_{k^{2}}(h')|_{k}\ldots \mu^{R}_{k^{m}}(h')|_{k}h$ (respect to  $\mu^{L}_{k^{j}}(-), j=1,\ldots, m$) hence $h'$ is the right (respect to left) end initial hyperbolic object we looking for, thanks to Part (2) of Lemma 3.22.
If there is no any objects sandwiched between $h$ and $h'$, then we have two cases: First case if $h|_{k}\mu^{R}_{k}(h)$ (respect to categorical $\mu^{L}_{k}(-)$) hence again using  Part (2) of Lemma 3.22 we have $h'$ is the right (respect to left) end initial object. The second case is when $\mu^{R}_{k}(h)|_{k}h$, then $\mu^{R}_{k}(h)$ (respect to $\mu^{L}_{k}(-)$) is also an initial object, again thanks to Part (2) of Lemma 3.22. Then we keep applying right categorical mutations until we obtain an object that finishes  the right (respect to left) side of the zigzag, in other words until we obtain an object $g$ such that  $g|_{k}\mu^{R}_{k}(g)|_{k}\mu^{R}_{k^{2}}(g)$ (respect to categorical $\mu^{L}_{k}(-)$) and in this case $g$ is the right (respect to left) end initial object. This process is finite due to Part (3) of Lemma 3.22. In case of $h$ is an initial object, then if $\mu^{R}_{k}(h)|_{k}\mu^{R}_{k^{2}}(h)$ (respect to $\mu^{L}_{k}(-)$) then $h$ is the right (respect to left) end initial object. If there are other initial objects that lies to right (respect left) of $h$ then apply right (respect to left) mutation on $h$ until we obtain an object $g$  such that  $g|_{k}\mu^{R}_{k}(g)|_{k}\mu^{R}_{k^{2}}(g)$ (respect to categorical $\mu^{L}_{k}(-)$) which would be the right (respect to left) end initial object.
\end{proof}

\begin{lem} Every mutation category $\mathcal{H(P)}$ contains a full subcategory that is  a clustered hyperbolic category.
\end{lem}
\begin{proof}
Fix $\mathcal{P}=(\Theta, \xi, \mathcal{A}\langle \Theta \rangle)$ to be the categorical preseed of $\mathcal{H(P)}$. Let $\mathcal{A}_{0}\langle \Theta \rangle$ be the full subcategory  of $\mathcal{A}\langle \Theta \rangle$ with  objects given by:\\ $(M, \eta)$ is an object in $\mathcal{A}_{0}\langle \Theta \rangle$ if and only if  $(\mu(\eta), M, \eta)$ is a right end initial object of the cluster class $[(M, \eta)]_{0}$. Now consider the categorical pressed $\mathcal{P}_{0}=(\Theta, \xi, \mathcal{A}_{0}\langle \Theta \rangle)$ and the category $\mathfrak{C}(\mathcal{P})$  with objects are all the hyperbolic objects of $\mathcal{P}_{0}$. The morphisms between any two objects $(g',M,g)$ and $ (h', M',h)$ in  $\mathfrak{C}(P)$ are given by
\begin{equation}
\nonumber Mor._{\mathfrak{C}(P)}((g',M,g), (h', M',h)):=Mor._{\mathcal{H}(P)}((g',M,g), (h', M',h)).
\end{equation}

So by definition, $\mathfrak{C}(\mathcal{P})$ is a full subcategory of $\mathcal{H}(\mathcal{P})$. The proof of $\mathfrak{C}({\mathcal{P}})$ is a clustered hyperbolic category is due to  the following Lemma.
\end{proof}

\begin{lem} For every hyperbolic object $(\gamma, M, \eta)$ in the category  $\mathfrak{C}(\mathcal{P})$ there exists a unique categorical preseed $\mathcal{S}$, mutationally equivalent to the initial categorical preseed $\mathcal{P}_{0}$, such that  $(\gamma, M, \eta)$ is an object in one of the the hyperbolic categories  $\mathcal{H}^{R}(\mathcal{S})$ or $\mathcal{H}^{L}(\mathcal{S})$.
\end{lem}
\begin{proof} Let $H(\mathcal{S})=Obj.H^{R}(\mathcal{S})\cup Obj.H^{L}(\mathcal{S})$. Suppose that there are two different categorical preseeds $\mathcal{S}$ and $\mathcal{S'}$ such that
\begin{equation}\label{}
  (\gamma, M, \eta) \in H(S)\cap H(\mathcal{S'}).
\end{equation}
Then we have two cases
\begin{description}
\item[Case 1] Since the hyperbolic object  $(\gamma, M, \eta)$ belongs to two hyperbolic categories of two different  categorical preseeds, then it can be tracked back to two different initial objects say $h$ and $h'$ in $\mathcal{P}_{0}$. Hence $(\gamma, M, \eta)$ can be obtained from   $h$ or $h'$ by applying two different sequences of mutations. Therefore  $h$ and $h'$ are related by a sequence of mutations which means that both of $h$ and $h'$ are right end initial objects in the same zigzag  which is a contradiction with the uniqueness of right end initial objects.

  \item[Case 2] The zigzag presentation of the cluster class of $(\gamma, M, \eta)$ contains two copies of $(\gamma, M, \eta)$ which means there is a non-trivial sequence of right (or left) mutations  $\mu^{R}_{i_{1}}\cdots \mu^{R}_{i_{j}}$ such that $\mu^{R}_{i_{1}}\cdots \mu^{R}_{i_{j}}(M, \eta)=(M, \eta)$, where $\eta=(\eta_{1}, \cdots, \eta_{n})$. Which leads to one of the following two types of contradictions.
      \begin{itemize}
        \item Without loss of generality, let $\mu^{R}_{i_{t}}$ be a repeated mutations even number of times, say $2m$ for some natural number $m$,  inside the sequence  $\mu^{R}_{i_{1}}\cdots \mu^{R}_{i_{j}}$. Then we must have
            \begin{equation}\label{}
         \nonumber \xi^{m}_{i_{j}, \Theta_{i_{t}} (M)} \eta_{i_{t}} \xi^{-m}_{\Theta_{i_{t}}, M}=id_{\Theta_{i_{t}} (M)}
            \end{equation}
            which means that for every  $w \in M,$
            \begin{equation}\label{}
         \nonumber   \eta_{i_{t}}(w) =\xi^{-m}_{i_{t}, \Theta_{i_{t}} (M)}id_{\Theta_{i_{t}} (M)}\xi^{m}_{\Theta_{i_{t}}, M}(w)=w
            \end{equation}
            which contradicts with the fact that $\eta_{i_{t}}$ is different from the identity morphism.
        \item  Again without loss of generality, let $\mu^{R}_{i_{t'}}$ be a repeated mutations odd number of times, say $2m+1$ for some natural number $m$,   inside the sequence  $\mu^{R}_{i_{1}}\cdots \mu^{R}_{i_{j}}$. Then we must have
            \begin{equation}\label{}
         \nonumber \xi^{m+1}_{\Theta_{i_{t'}}(M)} \eta^{-1}_{i_{t'}}  \xi^{-m}_{i_{t'}, \Theta_{i_{t'}} (M)}=id_{M}
            \end{equation}
            which means that for every  $w \in M,$
            \begin{equation}\label{}
         \nonumber   \eta^{-1}_{i_{t'}}(w) =\xi^{-m-1}_{\Theta_{i_{t'}}, M}id_{\Theta_{i_{t'}} (M)}\xi^{m}_{i_{t'}, \Theta_{i_{t'}} (M)}(w)=\xi^{-1}_{i_{t'}}(w)
            \end{equation}
            which contradicts with the fact that $\eta_{i_{t'}}$ is different from the  morphism $\xi_{i_{t'}}$.
      \end{itemize}

\end{description}
\end{proof}

\begin{defn} Let $\mathcal{P}$ be a categorical preseed with its initial categorical preseed $\mathcal{P}_{0}$.  The  category $\mathfrak{C}(\mathcal{P})$ defined in the Proof of Lemma 3.24  is called  the \emph{clustered hyperbolic category} of $\mathcal{P}$ and the objects of the initial categorical preseed  $\mathcal{P}_{0}$ are called initial objects.
\end{defn}
One can see that a  clustered hyperbolic category of $\mathcal{P}$ could be also defined using the left initial objects instead of the right initial objects.

\begin{rem} If $\mathcal{P}=(\Theta, \xi, \mathcal{A}\langle \Theta \rangle)$ is a categorical preseed such that the skew Laurent category  $\mathcal{A}\langle \Theta \rangle$ contains no mutationally equivalent objects, then  the category $\mathcal{H}(\mathcal{P})$ is identified with  the clustered hyperbolic category   $\mathfrak{C}(\mathcal{P})$. Furthermore, in such case every Laurent skew category  $\mathcal{A}\langle \Theta \rangle$ of any categorical preseed $\mathcal{S}$, that is mutationally equivalent to $\mathcal{P}$  contains no mutationally equivalent objects.
\end{rem}

 In  Theorem 3.28  we provide  a  clustered hyperbolic category that   arises from Weyl categorical preseed $\mathcal{W}$, see Examples 3.8 and  3.11. The theorem also provides a relation between the arising clustered hyperbolic category and the category of representations of the Weyl cluster algebra $\mathcal{H}(p)$.

 Let $p=(\alpha, \xi, \theta)$ be a Weyl pressed in the skew-field $\mathcal{D}$ and $\mathcal{H}(p)$ be its Weyl cluster algebra. Denote the category of representations of $\mathcal{H}(p)$ by $\mathcal{H}(p)$-mod. Let $\psi_{*}:\mathcal{H}(p)$-mod$\rightarrow \textit{R}-mod$, the forgettable functor that sends each $\mathcal{H}(p)$ module to itself forgetting about the action of the elements of the cluster set $\mathcal{X}(p)$. Consider the  category $\mathcal{A}_{p}=\psi_{*}(\mathcal{H}(p)$-mod). One can see that  $\mathcal{A}_{p}$ is an additive category as it is a full subcategory of the category $\textit{R}-mod$.

\begin{thm}  The Weyl  preseed $p=(\alpha, \xi, \theta)$ gives rise to a clustered hyperbolic  category $\mathfrak{C}(\mathcal{W}_{0})$  that is generated by (possibly infinitely many)  equivalent hyperbolic categories where  each one of them is equivalent to  $\mathcal{A}_{p}$.

\end{thm}

\begin{proof}
 Let   $\mathcal{W}_{0}=(\Theta, \xi, \mathcal{A}_{p}\langle\Theta \rangle)$ be a categorical preseed in the additive category $\mathcal{A}_{p}$ induced from $p$ same way as in  Example 3.8, where the set of objects of the skew Laurent category $\mathcal{A}_{p}\langle\Theta \rangle$ is given by $\{(M, \overline{\alpha}); M \ \text{is an object in} \ \mathcal{A}_{p}\}$. By definition of preseeds the elements of $\alpha=\{\alpha_{1}, \ldots, \alpha_{n}\}$ are algebraically independent then so are the morphisms $\{\overline{\alpha}_{1}, \ldots, \overline{\alpha}_{n}\}$ which means that non of the objects of $\mathcal{A}_{p}\langle\Theta \rangle$ are mutationally equivalent. Therefore the category $\mathfrak{C}(\mathcal{W}_{0})$ of all hyperbolic objects of the categorical preseed $\mathcal{W}_{0}$ is naturally a clustered hyperbolic category, thanks to Remark 3.27 and Remark 3.17.

One can see that the  algebras  $\mathcal{H}^{R}(s)$, $\mathcal{H}^{L}(s)$  and the Weyl cluster algebra $\mathcal{H}(p)$ are all isomorphic for every preseed $s$ that is  mutationally equivalent  to $p$, thanks to Part (3) of Theorem 2.13.  Hence the categories $\mathcal{H}(p)-mod, H^{R}(s)-mod$ and $ H^{L}(s)$-mod are equivalent. Also, one can define  a one to one correspondence between the objects of  $\mathcal{A}_{p}$ and the set of all cluster classes of $\mathfrak{C}(\mathcal{W}_{0})$  given by $M\leftrightarrow \overline{(M, \overline{\alpha})}$. Which gives rise to pairs of  inverse functors $\mathcal{F}_{s}$  and $\mathcal{G}_{s}$  for every preseed $s$, mutationally  equivalent to $p$ and for every categorical preseed $\mathcal{S}$ mutationally equivalent to $\mathcal{W}_{0}$ that is $\mathcal{S}=\mu(\mathcal{W}_{0})$ and $s=\mu'(p)$ for some sequence of categorical mutations $\mu$ and  sequence of preseeds mutations $\mu'$ corresponding to $\mu$, given by

\begin{equation}\label{}
   \nonumber \mathcal{F}_{s}: \mathcal{H}(s)-\text{mod}\rightarrow \mathfrak{C}^{R}(\mathcal{S}) \ \ \text{given by} \ \ \ W \mapsto (\mu^{R}(\overline{\beta}),W , \overline{\beta});
     \end{equation}
     and
     \begin{equation}\label{}
   \nonumber \mathcal{G}_{s}: \mathfrak{C}^{R}(\mathcal{S}) \rightarrow \mathcal{H}(s)-\text{mod} \ \ \text{given by} \ \ (\mu^{R}(\overline{\beta}),W , \overline{\beta}) \mapsto  W;
     \end{equation}
     both of $\mathcal{F}_{s}$  and $\mathcal{G}_{s}$ are identities on morphisms, here $\overline{\beta}=\mu(\overline{\alpha})$.   Since the category $\mathfrak{C}^{R}(S)$ (respect to $\mathfrak{C}^{L}(S)$) is a full subcategory of $\mathfrak{C}(\mathcal{W}_{0})$, then $\mathcal{A}_{p}$ is embedded in the clustered hyperbolic category $\mathfrak{C}(\mathcal{W}_{0})$.
Finally, the following diagram explains the relations between the categories  $\mathcal{A}_{p}, \mathcal{H}(s)-\text{mod},  \mathfrak{C}^{R}(S)$,  $\mathfrak{C}^{L}(S)$ and
$ \mathfrak{C}(\mathcal{W}_{0})$

\begin{equation}
\mathcal{A}_{p} \cong \mathcal{H}(s)-\text{mod} \cong \mathfrak{C}^{R}(S) \cong  \mathfrak{C}^{L}(S) \hookrightarrow  \mathfrak{C}(\mathcal{W}_{0}).
\end{equation}

\end{proof}

\begin{thm} Let $\mathcal{P}$ be a categorical preseed in $\mathcal{A}_{p}$ and $\mathfrak{C}$ be any clustered hyperbolic subcategory of $\mathcal{H}(\mathcal{P})$. Then there is a map $\mathbf{F}_{P}$ from  $Obj.\mathfrak{C}$  to the  Weyl cluster algebra $\mathcal{H}(p)$ such that image of $\mathbf{F}_{p}$ generates $\mathcal{H}(p)$.

\end{thm}
\begin{proof}  Let   $\mathbb{W}=(\Theta, \xi, \mathcal{A}_{p}\langle\Theta \rangle)$ be the initial categorical preseed  of $\mathfrak{C}$. Let $\mathcal{M}(\mathbb{W})$ be the set of all morphisms that appear in every possible skew Laurent object $(M, \overline{\beta})$ that can be obtain from some initial skew Laurent objects $(W, \beta)$ by applying some sequence of categorical mutations, where $\overline{\beta}=\{\overline{\beta}_{1},\ldots, \overline{\beta}_{n}\}$. Let $\{\alpha_{1}, \ldots, \alpha_{n}\}$ be the set of initial cluster variables of the Weyl preseed $p$.  We introduce the   \textit{step back map} $f_{p}:\mathcal{M}(\mathbb{W})\rightarrow \mathcal{X}(p)$  given by

 \begin{equation}\label{}
 f_{p}(\overline{\beta_{k}})= f_{p}(\mu_{k^{m}}(\beta_{k})) =\begin{cases} \mu^{R}_{k}(\mu_{k^{m}}(\alpha_{k})), & \text{ if } m< 0;\\
  1, & \text{ if } m=0;\\
         \mu^{L}_{k}(\mu_{k^{m}}(\alpha_{k})), & \text{ if } m> 0,
    \end{cases}
    \end{equation}
    where, without loss of generality, we assumed that $\overline{\beta_{k}}=\mu_{k^{m}}(\beta_{k})$ for some integer number $m$.
One can see that the step back map $f_{p}$ is well define, since for every skew Laurent object $( M, \overline{\beta})$ there is a unique categorical pressed $\mathcal{S}$ mutationally equivalent to $\mathbb{W}$ such that $( M, \overline{\beta})$ is an object in the skew Laurent category of $\mathcal{S}$, thanks to Lemma 3.25. Which guarantees the uniqueness of the sequence of categorical mutations that creates each $\overline{\beta_{k}}$ from $\beta_{k}$ for $k\in [1, n]$ and hence  the uniqueness of $f_{p}(\overline{\beta_{k}})$ for every $k\in [1, n]$. Now consider the  map
\begin{equation}\label{}
 \nonumber \mathbf{F}_{p}: Obj.\mathfrak{C} \rightarrow \mathcal{H}(p),
 \end{equation}
  such that
  \begin{equation}\label{}
  (\overline{\beta'}, M, \overline{\beta})\mapsto \prod^{n}_{i=1} f_{p}(\overline{\beta}_{i}),
 \end{equation}
where the product  is ordered as follows \ $\prod^{n}_{i=1} f_{p}(\overline{\beta}_{i})=f_{p}(\overline{\beta}_{1})\cdots f_{p}(\overline{\beta}_{n})$. Now we prove that every cluster variable in $\mathcal{X}(p)$ has a pre-image in Obj.$\mathfrak{C}$. Let $z$ be a cluster variable in $\mathcal{H}(p)$. Then there is a non-negative $q$ and $k\in [1, n]$ such that  $\mu_{k^{\pm q}}(\alpha_{k})=z$, where $(+)$ is for right mutation and $(-)$ is for the left. Consider the hyperbolic object $(\nu', M, \nu)$ where $\nu=(\beta_{1}, \ldots, \mu_{k^{\pm q\pm1}}(\beta_{k}),\ldots \beta_{n})$, hence
\begin{equation}\label{}
  \nonumber \mathbf{F}_{p}(\nu', M, \nu)=1\cdots\mu_{k^{\mp 1}}(\mu_{k^{\pm q \pm 1}}(\beta_{k}))\cdots1=z.
\end{equation}
 With a similar argument we can show that every cluster monomial of the form $z_{1}\cdots z_{k}, k\in[1,n]$ has a hyperbolic object in $\mathfrak{C}$ as its $\mathbf{F}_{p}$-pre-image.

\end{proof}

\begin{cor} In particular, there is a map from Obj.$\mathfrak{C}(\mathcal{W}_{0})$ (where $\mathfrak{C}(\mathcal{W}_{0})$ is the clustered hyperbolic category defined in Theorem 3.28) with its image generates $\mathcal{H}(p)$ and the set of hyperbolic objects
\begin{equation}\label{}
\nonumber  \{\mu_{k^{m}}(\overline{\alpha}', M, \overline{\alpha}) | m\in \mathbb{Z}\setminus\{-1, 0\}, k\in [1,n]\}
\end{equation}
 is in one to one correspondence with  the  set of all cluster variables $\mathcal{X}(p)$.

\end{cor}

\subsection*{Acknowledgments}
 This work had  been started while I was attending the Topics in Mathematics class taught by the late Alexander L. Rosenberg. Many thanks to Zongzhu Lin for suggesting the problem and for our valuable discussions. Special thanks to Brian Bischof for introducing me to hyperbolic categories. I would like to thank the anonymous referee for the valuable remarks.


\begin{thebibliography}
\normalfont\footnotesize
\bibitem{VB1}
V. V. Bavula, “Generalized Weyl algebras and their representations”, Algebra i Analiz, \textbf{4}:1 (1992), 75–97; English transl. in St. Petersburg
Math. J.\textbf{4} (1993), no. 1, 71–92.

\bibitem{BFZ}
A. Berenstein, S. Fomin and A. Zelevinsky, Cluster algebra III:
Upper bounds and double Bruhat cells, Duke Math. J.; math.
RT/0305434.

\bibitem{BFZ}
A. Berenstein, and A. Zelevinsky, Quantum  Cluster algebras, Adv. Math. \textbf{195} (2005), no. 2, 405-455.

\bibitem{BMRT}
A. B. Buan, R. Marsh, M. Reineke, I. Reiten and G. Todorov, Tilting theory and cluster combinatorics, Adv. Math. \textbf{204} (2006), no. 2, 572–618.

\bibitem{CC}
P. Caldero and F. Chapoton, Cluster algebras as Hall algebras of quiver representations, Comment. Math. Helv. \textbf{81} (2006), 595–616.
\bibitem{D}
J. Dixmier, Enveloping algebras, North Holland 1977.
\bibitem{FG1}
V. V. Fock, and A. B. Goncharov, Cluster Ensembles, Quantization and The Dilogarithm, Ann.Scient.Ec.Norm.Sup, \textbf{42}, 865-930, 2009.
\bibitem{FG2}
V. V. Fock, and A. B. Goncharov, Cluster Ensembles, Quantization and The Dilogarithm II: The intertwiner, arXive:math/0702398v1 [math.QA]13 Feb. 2007.
\bibitem{FG3}
V. V. Fock, and A. B. Goncharov, The quantum  Dilogarithm and Representations of Cluster varieties, arXive:math/0702398v6 [math.QA]21 Jul. 2008.



\bibitem{FZ02}
S. Fomin and A. Zelevinsky. Cluster algebras I. Foundations.
J.Amer.Math. Soc., \textbf{15(2)}:497-529(electronic), 2002.


\bibitem{FZ03}
S. Fomin and A. Zelevinsky. Cluster algebras II: Finite type classification, Invent. Math. \textbf{154} (2003), 63-121.

\bibitem{FZ04}
S. Fomin and A. Zelevinsky. Cluster algebras IV: Coefficients,
Compositio Mathematica \textbf{143} (2007)112-164.

\bibitem{HL10}
D. Hernandez, B. Leclerc, Cluster algebras and quantum affine algebras,  Duke Math. J. Volume \textbf{154}, Number 2 (2010), 265-341.



\bibitem{L1}
 B. Leclerc, Cluster algebras and representation theory, arXiv:1009.4552v1 [math.RT] 23 Sep 2010
\bibitem{LR}
V.A. Lunts and  A.L. Rosenberg, Kashiwara theorem for hyperbolic algebras, Preprint MPIM-1999-82, 1999.
\bibitem{O31}
O. Ore, Linear equations in non-commutative fields, Ann. of Math. (2)32 (1931)463-477.
\bibitem{R95}
A. L. Rosenberg, Algebraic Geometry and representations of quantized Algebras,  Kluwer academic  publishers, Dordrecht, Boston London, 1995.


\bibitem{Sa10}
I. A. Saleh, Cluster Structure on Generalized Weyl Algebras,  Algebr Represent Theor (2016)\textbf{ 19} :1017-1041.

\bibitem{Z04}
 A. Zelevinsky. Cluster algebras: Notes for 2004 IMCC (Chonju, Korea, August
 2004), arXiv:math.RT/0407414.



\end{thebibliography}
\end{document}